\documentclass[reqno]{amsart}
\usepackage{hyperref}
\usepackage{refcheck}

\begin{document}
\title[\hfilneg
\hfil asymptotically linear fractional Schr\"{o}dinger equations ]
{Positive solutions to some asymptotically linear fractional Schr\"{o}dinger equations}

\author[J. Zhang\,\,,\,\,X. Liu\hfil
\hfilneg]
{Jinguo Zhang,\quad Xiaochun Liu}

\address{Jinguo Zhang\newline
   School of Mathematics \\
  Jiangxi Normal University \\
  330022 Nanchang, China}
\email{yuanxin1027suda@163.com}

\address{Xiaochun Liu \newline
   School of Mathematics and Statistics \\
  Wuhan University \\
  430072 Wuhan, China}
\email{xcliu@whu.edu.cn}

\thanks{Supported by NSFC Grant No.11371282 and
Natural Science Foundation of Jiangxi (No. 20142BAB211002).}

\begin{abstract}
This paper is devoted to prove the existence and nonexistence of positive solutions for a class of fractional
Schr\"{o}dinger equation in $\mathbb{R}^{N}$ of the form
$$(-\Delta)^{s}u+V(x)\, u=f(u)£¬\quad x\in\mathbb{R}^{N},$$
where $N>2s$, $s\in (0\,,\,1)$ and $(-\Delta)^{s}$ stands for the fractional Laplacian.
We apply a new methods to obtain the existence of positive solutions when $f(u)$ is asymptotically
linear with respect to $u$ at infinity.

 \end{abstract}

\maketitle \numberwithin{equation}{section}
\newtheorem{theorem}{Theorem}[section]
\newtheorem{lemma}{Lemma}[section]
\newtheorem{remark}{Remark}[section]
\newtheorem{proposition}{Proposition}[section]
\newtheorem{corollary}{Corollary}[section]
\newtheorem{definition}{Definition}[section]
\newcommand{\R}{\mathbb{R}^{N}}
\newcommand{\HR}{H^{s}(\mathbb{R}^{N})}

\section{Introduction and main results}
In this paper,  we consider the existence of solutions to the nonlinear
 Schr\"{o}dinger equation with fractional Laplacian
\begin{align*}\label{eq3-1}
\quad\quad\qquad \qquad \qquad \qquad  (-\Delta)^{s}u+V(x)\, u=f(u),\quad x\in\mathbb{R}^{N}.
\quad\quad \quad\quad \qquad \qquad  \text{(FSE)}
\end{align*}
where $s\in (0\,,\,1)$ and  $(-\Delta)^{s}$ stands for the fractional Laplacian.
The potential $V\in C^{2}(\mathbb{R}^{N},\mathbb{R})$ and satisfies
\begin{itemize}
\item[$(V_{1})$] $V_{0}:=\inf\limits_{\mathbb{R}^{N}}V(x)>0$,\, $V_{\infty}=\lim\limits_{|x|\to +\infty}V(x)>0$;
\item[$(V_{2})$] $\langle\nabla V(x), x\rangle\leq 0$ for all $x\in \mathbb{R}^{N}$, where the strict inequality holds on a subset of
positive Lebesgue measure of $\mathbb{R}^{N}$;
\item[$(V_{3})$] $N\, V(x)+\langle\nabla V(x), x\rangle\geq N\,V_{\infty}$ for all $x\in \mathbb{R}^{N}$;
\item[$(V_{4})$] $\langle\nabla V(x), x\rangle+\frac{x\cdot H(x)\cdot x}{N}\leq 0$ for all $x\in \mathbb{R}^{N}$,
where $H$ represents the Hessian matrix of the function $V$.
\end{itemize}

This equation was introduced by Lashin \cite{nl2000,nl2002}, and comes from an expansion of the Feynman path integral from
Brownian-like to L\'{e}vy-like quantum mechanical paths. When $s=1$, the L\'{e}vy dynamics becomes the Brownian dynamics,
and equation (FSE) reduces to the classical Schr\"{o}dinger equation.

The fractional Schr\"{o}dinger equations is an important model in the study
of the fractional quantum mechanics. Recently, this has been widely investigated by many authors in the last decades,
 see \cite{cheng2012,cw2013,dpv2013,fqt2012,fv2013,mw2012,pp2013,s2012,secchi2013,sz2014} and references therein.
In most of the paper mentioned above the existence of positive solutions has been considered under different assumptions
on $V$ and $f$.

 In \cite{fqt2012}, Felmer et al. studied a similar class of equations, in which $V(x)\equiv1$ and the nonlinearity $f$
 has subcritical growth and satisfies the Ambrosetti-Rabinowitz condition, i.e., there exists $\theta>2$ such that
 \begin{equation}\label{eq1-1}
 0<\theta F(x,t)\leq tf(x,t),\quad \forall t>0,\,\,\,\text{a.e.}\,\,x\in \mathbb{R}^{N}.
 \end{equation}
 Using critical point theory, the classical positive solutions are found and some interesting
results on regularity are offered.
In \cite{secchi2013}, several existence results were proved for problem (FSE)
 with more general nonlinearities on the right hand side, see also \cite{cheng2012}.
Secchi\cite{secchi2013} obtained the existence of ground state solution
of (FSE) when $V(x)\to +\infty$ as $|x|\to +\infty$ when (AR)-condition \eqref{eq1-1} holds.
In \cite{dpv2013}, the authors looked for radially symmetric solutions of (FSE) when $V$
and $f$ do not depend explicitly on the space variable $x$.

As is well-known, the (AR)-condition \eqref{eq1-1} implies that the nonlinearity $f$ is superquadric at infinity,
roughly speaking,  $\frac{F(x,t)}{t^2}\to +\infty$ as $|t|\to +\infty$ uniformly in $x\in \mathbb{R}^{N}$,
where $F(x,t)=\int_{0}^{t}f(x,\tau)d\tau$. Note that, in some physical problems,
the nonlinearity term $f$ is asymptotically linear with respect to $t$ at infinity,
 which does not satisfy the (AR)-condition. In this paper, motivated by these consideration,
 the main features of problem (FSE) is that the nonlinearity is asymptotically linear and the associated
 problem at infinity is autonomous.

Throughout this paper, the functional $f:\,\mathbb{R}^{+}\to \mathbb{R}$ is continuous and satisfies:
\begin{itemize}
\item[$(f_{1})$]  $\lim\limits_{t\to 0^{+}}\frac{f(t)}{t}=0$ and $\lim\limits_{t\to +\infty}\frac{f(t)}{t}=1$;
\item[$(f_{2})$] There is a constant $L\in[1\,,\,+\infty)$ such that
$$0<Q(r)\leq L\,Q(t)$$
for all $0<r\leq t$, and $\lim_{|t|\to +\infty}Q(t)=+\infty$,
where $Q(t)=\frac{1}{2}t\,f(t)-F(t)$ and $F(t)=\int^{t}_{0}f(\tau)d\tau$.
\end{itemize}
As we have mentioned, we do not assume the (AR)-condition.

The energy functional associated with problem (FSE) is defined by
$$\mathcal{I}(u)=\frac{1}{2}\int\limits_{\mathbb{R}^{N}}|(-\Delta)^{\frac{s}{2}}u|^{2}dx
+\frac{1}{2}\int\limits_{\mathbb{R}^{N}}V(x)|u|^{2}dx
-\int\limits_{\mathbb{R}^{N}}F(u)dx.$$
Clearly, the condition $(f_1)$ implies that there are $\varepsilon>0$, $2\leq p\leq 2^*_{s}:=\frac{2N}{N-2s}$
and $C(\varepsilon)>0$
such that for all $u\in \HR$, we have
\begin{equation}\label{eq1-2}
|F(u)|\leq \frac{\varepsilon}{2}|u|^{2}+\frac{C(\varepsilon)}{p}|u|^{p}.
\end{equation}
Due to this observation, one can show that the energy functional
$\mathcal{I}$ is well-defined and belongs to $C^{1}(E,\mathbb{R})$.

To the best of our knowledge, we can not find any result in the literature that can be directly applied to our problem (FSE).
In order to deal with problem (FSE), one has to face various difficulties:
 firstly,  we mention that the lack of compactness of the embedding of $H^{s}(\mathbb{R}^{N})$
in the Lebesgue space $L^{p}(\mathbb{R}^{N})$ for $p\in (2\,,\,2^*_{s})$.
This prevents us from using the variational techniques in a standard way.
Secondly, since we does not have  symmetrical  assumptions, a basic step in the study of the Palais-Smale sequence for the
functional $\mathcal{I}$ is very difficult. For this we give a representation theorem for (PS)-sequence and
show that the only obstacle to the compactness is the solutions of the problem  at infinity:
\begin{align*}
\quad\quad\qquad \qquad \qquad \qquad  (-\Delta)^{s}u+V_{\infty}\, u=f(u)\quad \text{in}\,\,\,\mathbb{R}^{N}.
\quad\quad \quad\quad \qquad \qquad  \text{(FSE)}_{\infty}
\end{align*}
Moreover, some new estimates for the fractional elliptic problem are needed to be re-established.

The main results of this paper are stated as follows:
\begin{theorem}\label{th1-1}
Assume $(V_1)-(V_4)$ and $(f_1)-(f_2)$ hold. Then
 $$p=\inf\{\mathcal{I}(u):\,u\in \mathcal{P}\}$$
is not a critical value for the functional $\mathcal{I}$. In particular, the infimum $p$ is not achieved.
\end{theorem}

\begin{theorem}\label{th1-2}
Let $(V_1)-(V_4)$ and $(f_1)-(f_2)$ hold. Furthermore we assume
\begin{itemize}
\item[$(V_{5})$] $\frac{(N-2s)\|w\|^{2}_{H^{s}(\mathbb{R}^{N})}}{A\|w\|_{L^{2}(\mathbb{R}^{N})}^{2}}\leq 2^{\frac{2s}{N}}$,
 where $A:=\max\limits_{x\in \mathbb{R}^{N}}\langle \nabla V(x),x\rangle$.
\end{itemize}
Then problem \em{(FSE)} has at least one positive solution in $\HR$.
\end{theorem}

The purpose of this paper is to present a different and more general approach in the search for
fractional Schr\"{o}dinger equation with asymptotically linear nonlinearity.  Furthermore, the
condition $(f_2)$ is more general than the usual assumption that $\frac{f(t)}{t}$
is an increasing function of $t>0$. Then not all $u\in H^{s}(\mathbb{R}^{N})\setminus\{0\}$
can be projected on the Nehari manifold and so this approach fails.
Our interest in this paper is to search for the positive solution of Eq. (FSE) on the Pohozaev manifold $\mathcal{P}$.
Our ideas were inspired in the recent work of Secchi \cite{s2012}
which related the minimization of the functional $\mathcal{I}$ on the Pohozaev manifold $\mathcal{P}$.
The Pohozaev manifold constraint has been used for the first time in the finding solutions to
the elliptic problem  by Shatah \cite{sh1985}.  We refer the readers who are interested in the detail to the papers
\cite{lk2004,aa2009,lk2002,lla2014} and the reference therein.

The organization of this paper is as follows. In Section 2, we introduce a variational setting of
the problem and present some preliminary results.
In Section 3, some properties of the Pohozaev manifold are discussed.
In Section 4, we investigate the behavior of the Palais-Smale sequence for the functional $\mathcal{I}$.
Section 5 and 6 are devoted to the proof of Theorems \ref{th1-1} and \ref{th1-2}.

\section{Preliminaries}
In this section, we collect our basic assumptions and recall some known results for future reference.
In this paper, $\mathcal{S}$ denotes the Schwartz space of rapidly decreasing $C^{\infty}$ functions in $\mathbb{R}^{N}$
and $\mathcal{F}$ is the Fourier transform, i.e.,
$$\mathcal{F}[u](\xi)=\frac{1}{(2\pi)^{\frac{N}{2}}}\int\limits_{\mathbb{R}^{N}}e^{-2\pi i\,\xi\cdot x}u(x)dx.$$

For any $s\in (0\,,\,1)$, the fractional Sobolev space $H^{s}(\mathbb{R}^{N})$ is defined by
$$H^{s}(\mathbb{R}^{N})=\{u\in L^{2}(\mathbb{R}^{N}):\,\frac{|u(x)-u(y)|^{2}}{|x-y|^{N+2s}}\in L^{1}(\mathbb{R}^{N}\times\mathbb{R}^{N})\},$$
endowed with the natural morm
$$\|u\|_{H^{s}(\mathbb{R}^{N})}=\Big(
\int\limits_{\mathbb{R}^{2N}}\frac{|u(x)-u(y)|^{2}}{|x-y|^{N+2s}}dxdy+\int\limits_{\mathbb{R}^{N}}|u|^{2}dx\Big)^{\frac{1}{2}}, $$
where the norm
$$[u]_{H^{s}(\mathbb{R}^{N})}=\Big(\int\limits_{\mathbb{R}^{2N}}\frac{|u(x)-u(y)|^{2}}{|x-y|^{N+2s}}dxdy\Big)^{\frac{1}{2}}$$
is the so called Gagliardo semi-norm of $u$.

Indeed, according to \cite{ege2012}, the fractional Laplacian $(-\Delta)^{s}$ can be viewed as a pseudo-differential operator of  symbol
$|\xi|^{2s}$, as stated in the following.
\begin{lemma}\label{le2-1}
Let $(-\Delta)^{s}:\mathcal{S}\to L^{2}(\mathbb{R}^{N})$ be the fractional Laplacian operator defined by
$$(-\Delta)^{s}u(x)=C_{N,s}\,\text{\em {P.V.}}\int\limits_{\mathbb{R}^{N}}\frac{u(x)-u(y)}{|x-y|^{N+2s}}dy,$$
where \text{\em {P.V.}} is the principle value and $C_{N,s}>0$ is a normalization constant.
 Then for any $u\in \mathcal{S}$,
$$(-\Delta)^{s}u(x)=\mathcal{F}^{-1}[|\xi|^{2s}(\mathcal{F}u)](x),\quad\forall  \xi\in \mathbb{R}^{N}.$$
\end{lemma}

Now we can see that an alternative definition of the fractional Sobolev space $H^{s}(\mathbb{R}^{N})$
via the Fourier transform is the following
$$H^{s}(\mathbb{R}^{N})=\{u\in L^{2}(\mathbb{R}^{N}):\,\int\limits_{\mathbb{R}^{N}}
(1+|\xi|^{2s})|\hat{u}|^{2}d\xi<+\infty\},$$
and the norm is defined by
$$\|u\|_{H^{s}(\mathbb{R}^{N})}=\Big(\int\limits_{\mathbb{R}^{N}}(|\xi|^{2s}|\hat{u}|^{2}+|\hat{u}|^{2})d\xi\Big)^{\frac{1}{2}},$$
where $\hat{u}:=\mathcal{F}[u]$ denotes the Fourier transform of $u$.

It follows from the equality
$$2 C^{-1}_{N,s}\int\limits_{\mathbb{R}^{N}}|\xi|^{2s}|\hat{u}|^{2}d\xi
=2 C_{N,s}^{-1}\|(-\Delta)^{\frac{s}{2}}u\|^{2}_{L^{2}(\mathbb{R}^{N})}
=[u]^{2}_{H^{s}(\mathbb{R}^{N})}
$$
that the norms on $H^{s}(\mathbb{R}^{N})$
\begin{equation*}\aligned
&u\mapsto \Big(\int\limits_{\mathbb{R}^{2N}}\frac{|u(x)-u(y)|^{2}}{|x-y|^{N+2s}}dxdy+\int\limits_{\mathbb{R}^{N}}|u|^{2}dx\Big)^{\frac{1}{2}},\\
&u\mapsto \Big(\int\limits_{\mathbb{R}^{N}}(|(-\Delta)^{\frac{s}{2}}u|^{2}+|u|^{2})dx \Big)^{\frac{1}{2}},\\
&u\mapsto \Big(\int\limits_{\mathbb{R}^{N}}|\xi|^{2s}|\hat{u}|^{2}d\xi+\int\limits_{\mathbb{R}^{N}}|u|^{2}dx \Big)^{\frac{1}{2}},\\
\endaligned
\end{equation*}
are all equivalent, and then the problem will be considered under those norms.

For the reader's convenience, from \cite{ege2012,s2012,secchi2013},
we review the main embedding results for the fractional Sobolev space.
\begin{lemma}\label{le2-2}
Let $s\in (0\,,\,1)$ and $N>2s$. Then there exists a constant $C>0$ such that
$$\|u\|_{L^{2^{*}_{s}}(\mathbb{R}^{N})}\leq C\,\|u\|_{H^{s}(\mathbb{R}^{N})},$$
for every $u\in H^{s}(\mathbb{R}^{N})$, where $2^{*}_{s}:=\frac{2N}{N-2s}$ is the fractional critical exponent.
Moreover, the embedding $H^{s}(\mathbb{R}^{N})\hookrightarrow L^{p}(\mathbb{R}^{N})$
is continuous for any $p\in [2\,,\,2^{*}_{s}]$ and locally compact whenever $p\in [2\,,\,2^{*}_{s})$.
\end{lemma}

\begin{lemma}\label{le2-3}
Assume that  the sequence $\{u_{n}\}_{n}$ is bounded in $H^{s}(\mathbb{R}^{N})$ and it satisfies
$$\lim\limits_{n\to +\infty}\sup\limits_{y\in \mathbb{R}^{N}}\int\limits_{B_{R}(y)}|u_{n}(x)|^{2}dx=0$$
for $R>0$. Then $u_{n}\to 0$ in $L^{p}(\mathbb{R}^{N})$ as $n\to\infty$ for every $2<p<2^*_{s}$.
\end{lemma}

Let
$$
E=\Big\{u\in H^{s}(\mathbb{R}^N): \int\limits_{\mathbb{R}^N}V(x)|u|^2dx<+\infty\Big\}
$$
be the Sobolev space endowed with the norm
$$
\|u\|_E=\Big(\int\limits_{\mathbb{R}^N}|\xi|^{2s}|\hat{u}|^{2}d\xi +\int\limits_{\mathbb{R}^{N}}V(x)|u|^{2}dx\Big)^{\frac{1}{2}}.
$$
It follows from $(V_1)$ that the norm $\|\cdot\|_{E}$ is equivalent to the norms $\|\cdot\|_{H^{s}(\mathbb{R}^{N})}$.

\begin{definition}\label{def2-1}
We say $u\in E$ is a weak solution of Eq. \em {(FSE)} if
$$\int\limits_{\mathbb{R}^{N}}|\xi|^{2s}\hat{u}(\xi)\hat{v}(\xi)d\xi+\int\limits_{\mathbb{R}^{N}}V(x)uvdx
=\int\limits_{\mathbb{R}^{N}}f(u)vdx$$
for all $v\in E$.
\end{definition}
It is easy to see that the weak solutions to Eq. (FSE) are the critical points of the functional
$\mathcal{I}:\,E\to \mathbb{R}$ defined by
$$\mathcal{I}(u)=\frac{1}{2}\int\limits_{\mathbb{R}^{N}}|\xi|^{2s}|\hat{u}(\xi)|^{2}d\xi
+\frac{1}{2}\int\limits_{\mathbb{R}^{N}}V(x)|u|^{2}dx
-\int\limits_{\mathbb{R}^{N}}F(u)dx,$$
where $F(u)=\int^{u}_{0}f(t)dt$. From \eqref{eq1-2},
 the functional $\mathcal{I}$ is well-defined and belongs to $C^{1}(E\,,\,\mathbb{R})$.

Since we are looking for positive solutions, we take as usual $f(u)$
defined on all $u\in \mathbb{R}$, making $f(u)=0$ if $u\leq 0$.
Thus, if $u$ is a critical point of $\mathcal{I}$, we obtain
\begin{equation*}\label{eq1-7}
\aligned
0
&=\langle\mathcal{I}'(u)\,,\,u^{-}\rangle\\
&=\int\limits_{\mathbb{R}^{2N}}\frac{(u(x)-u(y))(u^{-}(x)-u^{-}(y))}{|x-y|^{N+2s}}dxdy
+\int\limits_{\mathbb{R}^{N}}V(x)u\,u^{-}dx-\int\limits_{\mathbb{R}^{N}}f(u)u^{-}dx\\
&\geq \int\limits_{\mathbb{R}^{2N}}\frac{|u^{-}(x)-u^{-}(y)|^{2}}{|x-y|^{N+2s}}dxdy
+\int\limits_{\mathbb{R}^{N}}V(x)u\,u^{-}dx-\int\limits_{\mathbb{R}^{N}}f(u)u^{-}dx\\
\endaligned
\end{equation*}
where $u^{-}=\min\{u\,,\,0\}$. This implies that
$$\|u^{-}\|^{2}_{E}=\int\limits_{\mathbb{R}^{2N}}\frac{|u^{-}(x)-u^{-}(y)|^{2}}{|x-y|^{N+2s}}dxdy=0.$$
Thus, necessarily we have $u\geq 0$.

Here we use the following inequality
\begin{equation}\label{eq1-8}
(u(x)-u(y))(u^{-}(x)-u^{-}(y))\geq |u^{-}(x)-u^{-}(y)|^{2}\quad \text{for any}\,\,\,x,y\in\mathbb{R}^{N},
\end{equation}
where $u^{-}(x)=\min\{u(x)\,,\,0\}$. To check \eqref{eq1-8}, since the role of $x$
and $y$ is symmetric, we can always suppose $u(x)\geq u(y)$.
Also \eqref{eq1-8} is clearly an identity when $x,\,y\in \{u>0\}$ and
when $x,\,y\in\{u\leq 0\}$. So it only remains to check \eqref{eq1-8} when
$x\in\{u>0\}$ and $y\in \{u\leq 0\}$. In this case we have
$$u^{-}(x)-u^{-}(y)=0-u^{-}(y)=-u(y)\leq  u(x)-u(y).$$
We multiply by $u^{-}(x)-u^{-}(y)=-u(y)\geq 0$ from both sides and obtain \eqref{eq1-8}.

We will need a version of the linking theorem with Palais-Smale condition,
which we state here for the sake of completeness.
\begin{definition}\label{def4-2}
Let $S$ be a closed subset of a Banach space $X$ and $Q$ a submanifold of $X$ with relative boundary
$\partial Q$. We say that $S$ and $\partial Q$ link if
\begin{itemize}
\item[{ (i)}] $S\cap \partial Q=\emptyset$;
\item[{ (ii)}] $h(Q)\cap S\neq \emptyset$ for any $h\in\mathcal{H}:=\{h\in  C(X,X):\,h|_{\partial Q}=\text{id}\}$.
\end{itemize}
\end{definition}

\begin{theorem}\label{th4-1}
Suppose that $I\in C^{1}(X\,,\,\mathbb{R})$ is a functional satisfying (PS)-condition.
Consider a closed subset $S\subset X$
and a submanifold $Q\subset X$ with relative boundary $\partial Q$. Suppose
also that
\begin{itemize}
\item[{(i)}] $S$ and $\partial Q$ link;
\item[{(ii)}] $\alpha_{S}:=\inf_{u\in S}I(u)>\sup_{u\in \partial Q}I(u):=\alpha_{\partial Q}$;
\item[{(iii)}] $\sup\limits_{u\in Q}<+\infty$.
\end{itemize}
Then the real number
$$\alpha=\inf\limits_{h\in \mathcal{H}}\sup\limits_{u\in Q}I(h(u))$$
defines a critical point of $I$ with $\alpha\geq \alpha_{S}$.
\end{theorem}

\section{The Pohozaev identity  and some properties}

To solve problem {(FSE)}, we will look for critical points of the functional $\mathcal{I}$. In this
section we will show some properties of the Pohozaev manifold. The following result is
sketched in the paper \cite{s2012}.
\begin{proposition}\label{p2-1}
Let $u\in H^{s}(\mathbb{R}^{N})$ be a weak solution of Eq.\em{(FSE)}. Then $u$ satisfies
\begin{equation}\label{eq2-1}
\aligned
\frac{N-2s}{2}\int\limits_{\mathbb{R}^{N}}|\xi|^{2s}|\hat{u}(\xi)|^{2}d\xi+
&\frac{1}{2}\int\limits_{\mathbb{R}^{N}}\langle\nabla V(x), x\rangle|u|^{2}dx\\
&+\frac{N}{2}\int\limits_{\mathbb{R}^{N}}V(x)|u|^{2}dx=
N\int\limits_{\mathbb{R}^{N}}F(u)dx,\\
\endaligned
\end{equation}
where $F(u)=\int_{0}^{u}f(t)dt$.
\end{proposition}

Proposition \ref{p2-1} is obtained just by collecting the results in \cite{dpv2013,s2012,xj2012}.
The identity \eqref{eq2-1} is called ``Pohozaev identity". In this paper, we assume
that $\frac{f(u)}{u}$ is nondecreasing, and the path $m(t):=\mathcal{I}(t\,u)$ may not
intersect with the Nehari manifold
$$\mathcal{N}=\{u\in H^{s}(\mathbb{R}^{N})\setminus\{0\}:\,\langle\mathcal{I}'(u)\,,\,u\rangle=0\}$$
for a unique $t$. Indeed, it may happen that it does not intersect with the Nehari manifold at all or intersects with
the Nehari manifold at infinitely many points. This is the main reason why we are led to look for a
different approach using the Pohozaev manifold.

Now, we define the Pohozaev set associated with {(FSE)} by
$$\mathcal{P}:=\{u\in \HR\setminus\{0\}:\,\,\mathcal{J}(u)=0\},$$
where the functional $\mathcal{J}:\,\HR\to \mathbb{R}$ is defined by
\begin{equation*}\label{eq2-2}
\aligned
\mathcal{J}(u)&=\frac{N-2s}{2}\int\limits_{\mathbb{R}^{N}}|\xi|^{2s}|\hat{u}(\xi)|^{2}d\xi+
\frac{1}{2}\int\limits_{\mathbb{R}^{N}}\langle\nabla V(x), x\rangle|u|^{2}dx\\
&+\frac{N}{2}\int\limits_{\mathbb{R}^{N}}V(x)|u|^{2}dx
-N\int\limits_{\mathbb{R}^{N}}F(u)dx.\\
\endaligned\end{equation*}

In the following we shall show some properties for the functional $\mathcal{J}$ and the set $\mathcal{P}$.

\begin{proposition}\label{l2-1}
Assume that $(f_1)-(f_2)$ hold. Then $\mathcal{P}$ is a complete $C^{1}$-manifold.
\end{proposition}

\begin{proof}[\bf Proof.]
Clearly, $\mathcal{J}$ is continuous on $H^{s}(\mathbb{R}^{N})$.
We claim that $\mathcal{P}$ is a closed subset in $H^{s}(\mathbb{R}^{N})$.
Indeed, let $\{u_{n}\}_{n}$ be a sequence in $\mathcal{P}$ such that
 $u_{n}\to u$ in $\mathcal{P}$ as $n\to \infty$.
As $\mathcal{J}$ is continuous, we get
\begin{equation*}
\aligned
\frac{N-2s}{2}\int\limits_{\mathbb{R}^{N}}|\xi|^{2s}|\hat{u}(\xi)|^{2}d\xi+
&\frac{1}{2}\int\limits_{\mathbb{R}^{N}}\langle\nabla V(x), x\rangle|u|^{2}dx\\
&+\frac{N}{2}\int\limits_{\mathbb{R}^{N}}V(x)|u|^{2}dx=
N\int\limits_{\mathbb{R}^{N}}F(u)dx,\\
\endaligned
\end{equation*}

On the other hand, it follows from $(V_1)$, $(V_3)$ and \eqref{eq1-2} that for some sufficiently small $\varepsilon>0$,
\begin{equation*}
\aligned
\frac{N-2s}{2}\|u_{n}\|^{2}_{E}
&<\frac{N-2s}{2}\int\limits_{\mathbb{R}^{N}}|\xi|^{2s}|\hat{u}_{n}(\xi)|^{2}d\xi
+\frac{N}{2}\int\limits_{\mathbb{R}^{N}}V(x)|u_{n}|^{2}dx\\
&=N\int\limits_{\mathbb{R}^{N}}F(u_{n})dx-\frac{1}{2}\int\limits_{\mathbb{R}^{N}}\langle \nabla V(x),x\rangle|u_{n}|^{2}dx\\
&\leq N\int\limits_{\mathbb{R}^{N}}F(u_{n})dx+N\int\limits_{\mathbb{R}^{N}}(V(x)-V_{\infty})|u_{n}|^{2}dx\\
&\leq \varepsilon\|u_{n}\|^{2}_{L^{2}(\mathbb{R}^{N})}+C(\varepsilon)\|u_{n}\|^{2^*_s}_{L^{2^*_s}(\mathbb{R}^{N})},\\
\endaligned
\end{equation*}
 which implies that there exists some constant $C>0$ such that
 $\|u_{n}\|^{2}_{E}\leq C\|u_{n}\|^{2^*_{s}}_{E}$¡£
  Hence, we deduce that $\|u_{n}\|_{E}\geq C_{1}>0$,
and
 \begin{equation}\label{eq2-3}
 \|u\|_{E}=\lim\limits_{n\to +\infty}\|u_{n}\|_{E}\geq C_{1}>0.
 \end{equation}
Thus, the claim is true.

To prove that $\mathcal{P}$ is a complete $C^{1}$ manifold, it is sufficient to show that
$\mathcal{J}'(u)$ is surjective on $\mathcal{P}$ and its kernel splits.
Obviously, by the assumptions of $F$ and $V$, we obtain that $\mathcal{J}$ is $C^{1}$ and
\begin{equation}\label{eq2-3}
\aligned
\langle \mathcal{J}'(u)\,,\,u\rangle=&(N-2s)\int\limits_{\mathbb{R}^{N}}|\xi|^{2s}|\hat{u}(\xi)|^{2}d\xi+
\int\limits_{\mathbb{R}^{N}}\langle\nabla V(x), x\rangle|u|^{2}dx\\
&+N\int\limits_{\mathbb{R}^{N}}V(x)|u|^{2}dx-
N\int\limits_{\mathbb{R}^{N}}f(u)u dx.\\
\endaligned\end{equation}

Moreover, for any  $u\in \mathcal{P}$, one has
\begin{equation}\label{eq2-4}
\aligned
(N-2s)\int\limits_{\mathbb{R}^{N}}|\xi|^{2s}|\hat{u}(\xi)|^{2}d\xi
&+\int\limits_{\mathbb{R}^{N}}\langle\nabla V(x), x\rangle|u|^{2}dx\\
&+N\int\limits_{\mathbb{R}^{N}}V(x)|u|^{2}dx=
2N\int\limits_{\mathbb{R}^{N}}F(u)dx.\\
\endaligned\end{equation}
Then, from \eqref{eq2-4}, \eqref{eq2-3} and $(f_2)$, we get
\begin{equation*}\label{eq2-5}
\aligned
\langle \mathcal{J}'(u)\,,\,u\rangle
&=2N\int\limits_{\mathbb{R}^{N}}F(u)dx-N\int\limits_{\mathbb{R}^{N}}f(u)u dx\\
&=N\int\limits_{\mathbb{R}^{N}}\Big(2F(u)-f(u)u\Big)dx\\
&< 0.
\endaligned
\end{equation*}
This shows that $\mathcal{J}'(u)$ is surjective.

Moreover, by $(V_3)$ we have,
\begin{equation*}\label{eq2-6}
\aligned
\mathcal{J}(u)
&=\frac{N-2s}{2}\int\limits_{\mathbb{R}^{N}}|\xi|^{2s}|\hat{u}(\xi)|^{2}d\xi+
\frac{N}{2}\int\limits_{\mathbb{R}^{N}}\frac{\langle\nabla V(x), x\rangle+NV(x)}{N}|u|^{2}dx\\
&-N\int\limits_{\mathbb{R}^{N}}F(u)dx\\
&\geq \frac{N-2s}{2}\int\limits_{\mathbb{R}^{N}}|\xi|^{2s}|\hat{u}(\xi)|^{2}d\xi
+\frac{N}{2}\int\limits_{\mathbb{R}^{N}}V_{\infty}|u|^{2}dx
-N\int\limits_{\mathbb{R}^{N}}F(u)dx\\
&\geq \frac{N-2s}{2}\|u\|^{2}_{H^{s}(\mathbb{R}^{N})}-N\frac{\varepsilon}{2}\|u\|^{2}_{L^{s}(\mathbb{R}^{N})}
-N\frac{C(\varepsilon)}{p}\|u\|^{p}_{L^{p}(\mathbb{R}^{N})}\\
&\geq \frac{1}{2}(N-2s-NC_{1}\varepsilon)\|u\|^{2}_{H^{s}(\mathbb{R}^{N})}-NC_{2}\frac{C(\varepsilon)}{p}\|u\|^{p}_{H^{s}(\mathbb{R}^{N})}.\\
\endaligned
\end{equation*}

If we take $\varepsilon>0$ sufficiently small and $\rho>0$ such that $N-2s-NC_{1}\varepsilon>0$
and $\rho^{p}<\frac{p}{4NC_{2}C(\varepsilon)}(N-2s-NC_{1}\varepsilon)\rho^{2}$,
then if taking $\|u\|_{H^{s}(\mathbb{R}^{N})}=\rho$, we get
$\mathcal{J}(u)\geq \frac{1}{4}(N-2s-NC_{1}\varepsilon)\rho^{2}>0$
and
\begin{equation}\label{eq2-7}
\mathcal{J}(u)>0\,\,\,\text{ for all}\,\,\, 0<\|u\|_{H^{s}(\mathbb{R}^{N})}<\rho.
\end{equation}
Moreover, from $\mathcal{P}\cup \{0\}=\mathcal{J}^{-1}(\{0\})$ and \eqref{eq2-7}, we obtain $\{0\}$
is an isolated point in $\mathcal{J}^{-1}(\{0\})$.
This completes the proof.

\end{proof}

\begin{proposition}\label{p2-2}
Assume that $(V_4)$ and $(f_1)-(f_2)$ hold.
Then $\mathcal{P}$ is a natural constraint of $\mathcal{I}$.
\end{proposition}

\begin{proof}[\bf Proof.]
If $u$ be a critical point of the functional $\mathcal{I}$ on the manifold $\mathcal{P}$,
then $u$ is a solution of the optimization problem
$$\text{minimize}\,\,\mathcal{I}(u)\,\,\,\text{subject to}\,\,\,\mathcal{J}(u)=0.$$
Hence, by the theory of Lagrange multipliers, there exists $\mu \in\mathbb{R}$ such that
$\mathcal{I}'(u)+\mu \mathcal{J}'(u)=0$.  Thus
$$\langle \mathcal{I}'(u)\,,\,u\rangle+\mu \langle \mathcal{J}'(u)\,,\,u\rangle=0,$$
implies
\begin{equation*}\label{eq3-36}
\aligned
0
&=\int\limits_{\mathbb{R}^{N}}|\xi|^{2s}|\hat{u}(\xi)|^{2}d\xi
+\int\limits_{\mathbb{R}^{N}} V(x)|u|^{2}dx
-\int\limits_{\mathbb{R}^{N}}f(u)udx
+\mu\Big((N-2s)\int\limits_{\mathbb{R}^{N}}|\xi|^{2s}|\hat{u}(\xi)|^{2}d\xi\\
&+N\int\limits_{\mathbb{R}^{N}}V(x)|u|^{2}dx
+\int\limits_{\mathbb{R}^{N}}\langle\nabla V(x), x\rangle\,|u|^{2}dx
-N\int\limits_{\mathbb{R}^{N}}f(u)udx\Big),\\
\endaligned
\end{equation*}
which can be rewritten as
\begin{equation*}\label{eq3-36}
\aligned
\int\limits_{\mathbb{R}^{N}}\Big[1+\mu(N-2s)\Big]|\xi|^{2s}|\hat{u}(\xi)|^{2}&d\xi
+\int\limits_{\mathbb{R}^{N}}\Big[(1+\mu N) V(x)\\
&+\mu \langle\nabla V(x), x\rangle\Big]|u|^{2}dx
=\int\limits_{\mathbb{R}^{N}}\Big[1+\mu N\Big]f(u)udx.\\
\endaligned
\end{equation*}

This expression may be associated with the equation
\begin{equation}\label{eq3-37*}
(1+\mu (N-2s))(-\Delta)^{s}u+[(1+\mu N) V(x)+\mu \langle\nabla V(x), x\rangle]u=(1+\mu N)f(u).
\end{equation}
Therefore,  $u$ is the solution of Eq. \eqref{eq3-37*}.
From Proposition \ref{p2-1}, we know that  $u$ satisfies
the Pohpzaev identity $\widetilde{\mathcal{J}}(u)=0$,
where
\begin{equation}\label{eq3-39}
\aligned
\widetilde{\mathcal{J}}(u)
&=\Big(1+\mu (N-2s)\Big)\frac{N-2s}{2}\int\limits_{\mathbb{R}^{N}}|\xi|^{2s}|\hat{u}(\xi)|^{2}d\xi\\
&+\frac{N}{2}\int\limits_{\mathbb{R}^{N}}\Big((1+\mu N) V(x)+\mu \langle\nabla V(x), x\rangle\Big)|u|^{2}dx\\
&+\frac{1}{2}\int\limits_{\mathbb{R}^{N}}\Big((1+\mu N) \langle\nabla V(x), x\rangle+\mu \,\,x\cdot H(x)\cdot x\Big)|u|^{2}dx\\
&-N(1+\mu N)\int\limits_{\mathbb{R}^{N}}F(u)dx.\\
\endaligned\end{equation}
This means that $u$ is in the following Pohozaev manifold
 \begin{equation*}\label{eq3-38}
\widetilde{\mathcal{P}}=\{u\in \HR\setminus\{0\}:\,\,\widetilde{\mathcal{J}}(u)=0\}.
\end{equation*}

On the other hand, recalling that $u\in \mathcal{P}$ and substituting \eqref{eq2-4} into \eqref{eq3-39},
it follows that
\begin{equation*}\label{eq3-40}
\aligned
\widetilde{\mathcal{J}}(u)
&=\Big(1+\mu (N-2s)\Big)\frac{N-2s}{2}\int\limits_{\mathbb{R}^{N}}|\xi|^{2s}|\hat{u}(\xi)|^{2}d\xi
+\frac{1+\mu N}{2}N\int\limits_{\mathbb{R}^{N}}V(x)|u|^{2}dx\\
&+\frac{1+2\mu N}{2}\int\limits_{\mathbb{R}^{N}}\langle\nabla V(x), x\rangle|u|^{2}dx
+\frac{\mu}{2}\int\limits_{\mathbb{R}^{N}} x\cdot H(x)\cdot x\,\,|u|^{2}dx\\
&-N(1+\mu N)\int\limits_{\mathbb{R}^{N}}F(u)dx\\
&=-2\mu \,s\frac{N-2s}{2}\int\limits_{\mathbb{R}^{N}}|\xi|^{2s}|\hat{u}(\xi)|^{2}d\xi
+\frac{\mu N}{2}\int\limits_{\mathbb{R}^{N}}\langle\nabla V(x), x\rangle|u|^{2}dx\\
&+\frac{\mu}{2}\int\limits_{\mathbb{R}^{N}} x\cdot H(x)\cdot x\,|u|^{2}dx.
\endaligned\end{equation*}
Since $ \widetilde{\mathcal{J}}(u)=0$ and $(V_4)$, it  yields that
$$0\leq \mu\,s(N-2s)\int\limits_{\mathbb{R}^{N}}|\xi|^{2s}|\hat{u}(\xi)|^{2}d\xi=
\frac{\mu N}{2}\int\limits_{\mathbb{R}^{N}}\Big(\langle\nabla V(x), x\rangle
+ \frac{x\cdot H(x)\cdot x}{N}\Big)|u|^{2}dx\leq 0.$$
Therefore, $\mu=0$, and  $I'(u)+\mu J'(u)=0$ becomes $I'(u)=0$,
which implies that $u$ is a critical point of $I$.
\end{proof}

\begin{remark}\label{r3-1}
Note that in the proof of Proposition \ref{p2-2} we only use the assumption $(V_4)$.
When the function $V(x)\equiv V_{\infty}$, the similar result will be obtained.
\end{remark}

Set
$$m:=inf\{\mathcal{I}(u):\,u\in \mathcal{P}\}.$$
As a consequence of \eqref{eq2-1} and $(V_2)$,  for any $u\in \mathcal{P}$, we have
\begin{equation*}\aligned
\mathcal{I}(u)
&=\frac{1}{2}\int\limits_{\mathbb{R}^{N}}|\xi|^{2s}|\hat{u}(\xi)|^{2}d\xi
+\frac{1}{2}\int\limits_{\mathbb{R}^{N}}V(x)|u|^{2}dx-\int\limits_{\mathbb{R}^{N}}F(u)dx\\
&=\frac{1}{2}\int\limits_{\mathbb{R}^{N}}|\xi|^{2s}|\hat{u}(\xi)|^{2}d\xi
+\frac{1}{2}\int\limits_{\mathbb{R}^{N}}V(x)|u|^{2}dx-
\Big(\frac{N-2s}{2N}\int\limits_{\mathbb{R}^{N}}|\xi|^{2s}|\hat{u}(\xi)|^{2}d\xi\\
&+\frac{1}{2}\int\limits_{\mathbb{R}^{N}}V(x)|u|^{2}dx+\frac{1}{2N}\int\limits_{\mathbb{R}^{N}}\langle \nabla V(x),x\rangle|u|^{2}dx\Big)\\
&=\frac{s}{N}\int\limits_{\mathbb{R}^{N}}|\xi|^{2s}|\hat{u}(\xi)|^{2}d\xi
-\frac{1}{2N}\int\limits_{\mathbb{R}^{N}}\langle \nabla V(x),x\rangle|u|^{2}dx\\
&\geq \frac{s}{N}\int\limits_{\mathbb{R}^{N}}|\xi|^{2s}|\hat{u}(\xi)|^{2}d\xi>0,
\endaligned\end{equation*}
that is,  it turns out that $m$ is a positive number.

Due to condition $(V_{1})$, the Eq.(FSE) becomes the autonomous problem at infinity, i.e.,

\begin{align*}\label{eq3-1}
\quad\quad\qquad \qquad \qquad \qquad  (-\Delta)^{s}u+V_{\infty}\, u=f(u)\quad \text{in}\,\,\,\mathbb{R}^{N}.
\quad\quad \quad\quad \qquad \qquad  \text{(FSE)}_{\infty}
\end{align*}

In this case we use the notation $\mathcal{J}_{\infty}(u)$ and $\mathcal{P}_{\infty}$,
respectively, for the functional and the natural constraint, namely,
$$\mathcal{J}_{\infty}(u)=\frac{N-2s}{2}\int\limits_{\mathbb{R}^{N}}|\xi|^{2s}|\hat{u}(\xi)|^{2}d\xi+
\frac{N}{2}\int\limits_{\mathbb{R}^{N}}V_{\infty} |u|^{2}dx-N\int\limits_{\mathbb{R}^{N}}F(u)dx,$$

$$\mathcal{P}_{\infty}=\{u\in \HR\setminus\{0\}:\,\mathcal{J}_{\infty}(u)=0\}.$$

Similar as in the proof of the Proposition \ref{p2-2},
we have that $\mathcal{P}_{\infty}$ is a natural constraint of $\mathcal{I}_{\infty}$,
and so there exists
$$m_{\infty}:=inf\{\mathcal{I}_{\infty}(u):\,u\in \mathcal{P}_{\infty}\}.$$
We state, in the following propositions, some results about $\mathcal{P}$,
$\mathcal{P}_{\infty}$, $m$ and $m_{\infty}$.

\begin{proposition}\label{p2-3}
Eq. $(\text{FSE})_{\infty}$ has a positive ground state solution $w\in H^{s}(\mathbb{R}^{N})$,
which is radially symmetric about the origin and unique up to a translations,
\end{proposition}

The proof of Proposition \ref{p2-3} can be seen in \cite{dpv2013}.
Since $w$ is a ground state solution,
we have $\mathcal{I}_{\infty}(v)\geq \mathcal{I}_{\infty}(w)=m_{\infty}$
for all solutions $v$ of $\text{(FSE)}_{\infty}$.

\vspace{5mm}

From the idea in \cite{s2012}, we can
define the open set $$\mathcal{O}:=\{u\in \HR\setminus\{0\}:\,\mathcal{F}(u)>0\},$$
where $\mathcal{F}(u)=\int_{\mathbb{R}^{N}}(2F(u)-V_{\infty}|u|^{2})\,dx$.
Then we have the following result.

\begin{proposition}\label{p3-1}
For each $u\in \mathcal{O}$, there exists a unique $t_{u}>0$ such that
$$u(\frac{x}{t_{u}})\in \mathcal{P}.$$
Moreover, the function $u\mapsto t_{u}$ such that $u(\frac{x}{t_{u}})\in \mathcal{P}$ is continuous.
\end{proposition}

\begin{proof}[\bf Proof.]
Let $u\in \mathcal{O}$. We define the function $\theta:\,(0,+\infty)\to \mathbb{R}$ by
\begin{equation*}\label{eq3-5}
\theta(t):=\mathcal{I}(u(\frac{x}{t}))
=\frac{t^{N-2s}}{2}\int\limits_{\mathbb{R}^{N}}|\xi|^{2s}|\hat{u}(\xi)|^{2}d\xi
+\frac{t^{N}}{2}\int\limits_{\mathbb{R}^{N}}V(t\, x)|u|^{2}dx
-t^{N}\int\limits_{\mathbb{R}^{N}}F(u)dx.
\end{equation*}
The derivative of $\theta$ is the following
\begin{equation}\label{eq3-6}
\aligned
\theta'(t)
&=\frac{N-2s}{2}t^{N-2s-1}\int\limits_{\mathbb{R}^{N}}|\xi|^{2s}|\hat{u}(\xi)|^{2}d\xi
+\frac{Nt^{N-1}}{2}\int\limits_{\mathbb{R}^{N}}V(t\, x)|u|^{2}dx\\
&+\frac{t^{N}}{2}\int\limits_{\mathbb{R}^{N}}\langle\nabla V(t\, x), x\rangle\,|u|^{2}dx
-Nt^{N-1}\int\limits_{\mathbb{R}^{N}}F(u)dx\\
&=t^{N-2s-1}\Big(\frac{N-2s}{2}\int\limits_{\mathbb{R}^{N}}|\xi|^{2s}|\hat{u}(\xi)|^{2}d\xi
+\frac{Nt^{2s}}{2}\int\limits_{\mathbb{R}^{N}}V(t\, x)|u|^{2}dx\\
&+\frac{t^{2s}}{2}\int\limits_{\mathbb{R}^{N}}\langle\nabla V(t x), t\,x\rangle\,|u|^{2}dx
-Nt^{2s}\int\limits_{\mathbb{R}^{N}}F(u)dx
\Big)\\
&=t^{N-2s-1}\Big\{\frac{N-2s}{2}\int\limits_{\mathbb{R}^{N}}|\xi|^{2s}|\hat{u}(\xi)|^{2}d\xi\\
&-\frac{Nt^{2s}}{2}\int\limits_{\mathbb{R}^{N}}\Big(2F(u)
-\frac{NV(tx)+\langle\nabla V(t x), t\,x\rangle}{N}\,|u|^{2}\Big)dx\Big\}.\\
\endaligned
\end{equation}

On the other hand, we have
\begin{equation}\label{eq3-7}
\aligned
\mathcal{J}(u(\frac{x}{t}))
&=\frac{N-2s}{2}t^{N-2s}\int\limits_{\mathbb{R}^{N}}|\xi|^{2s}|\hat{u}(\xi)|^{2}d\xi+
\frac{t^{N}}{2}\int\limits_{\mathbb{R}^{N}}\langle\nabla V(t x), t x\rangle|u|^{2}dx\\
&+\frac{Nt^{N}}{2}\int\limits_{\mathbb{R}^{N}}V(tx)|u|^{2}dx-
Nt^{N}\int\limits_{\mathbb{R}^{N}}F(u)dx\\
&=t^{N-2s}\Big(\frac{N-2s}{2}\int\limits_{\mathbb{R}^{N}}|\xi|^{2s}|\hat{u}(\xi)|^{2}d\xi+
\frac{t^{2s}}{2}\int\limits_{\mathbb{R}^{N}}\langle\nabla V(t x), t x\rangle|u|^{2}dx\\
&+\frac{Nt^{2s}}{2}\int\limits_{\mathbb{R}^{N}}V(tx)|u|^{2}dx-
Nt^{2s}\int\limits_{\mathbb{R}^{N}}2F(u)dx\Big)\\
&=t^{N-2s-1}\Big\{\frac{N-2s}{2}\int\limits_{\mathbb{R}^{N}}|\xi|^{2s}|\hat{u}(\xi)|^{2}d\xi\\
&-\frac{Nt^{2s}}{2}\int\limits_{\mathbb{R}^{N}}\Big(2F(u)
-\frac{NV(tx)+\langle\nabla V(t x), t\,x\rangle}{N}\,|u|^{2}\Big)dx\Big\}.\\
\endaligned\end{equation}
Taking account of \eqref{eq3-6} and \eqref{eq3-7},
we infer $u(\frac{x}{t})\in \mathcal{P}$ if and only if
$\theta'(t)=0$ for some $t>0$.

Note that by conditions $(V_1)$, $(V_2)$, $(V_3)$  and the Lebesgue Dominated Convergence Theorem, we have
\begin{equation*}\label{eq3-9}
\lim\limits_{t\to +\infty} \langle\nabla V(t x), tx\rangle=0,
\end{equation*}
and
\begin{equation}\label{eq3-8}
\lim\limits_{t\to +\infty}\int\limits_{\mathbb{R}^{N}}\Big(2F(u)dx-V(t x)|u|^{2}\Big)dx
=\int\limits_{\mathbb{R}^{N}}\Big(2F(u)dx-V_{\infty}|u|^{2}\Big)dx>0.
\end{equation}
Then we get
\begin{equation}\label{eq3-10}
\theta'(t)
=t^{N-2s-1}\Big\{\frac{N-2s}{2}\int\limits_{\mathbb{R}^{N}}|\xi|^{2s}|\hat{u}(\xi)|^{2}d\xi
-\frac{Nt^{2s}}{2}\int\limits_{\mathbb{R}^{N}}\Big(2F(u)-V_{\infty}|u|^{2}\Big)dx
\Big\}+o(1).
\end{equation}
Therefore, combining \eqref{eq3-8} with \eqref{eq3-10},
 we have $\theta'(t)<0$ for $t$ sufficiently large enough for each $u\in \mathcal{O}$.

Moreover, by $(V_1)$, $(V_3)$ and $F(u)>0$, we obtain that
\begin{equation}\label{eq3-12}
\aligned
-\max\limits_{x\in \mathbb{R}^{N}}|V|\int\limits_{\mathbb{R}^{N}}|u|^{2}dx
&<\int\limits_{\mathbb{R}^{N}}\Big(2F(u)-\frac{NV(tx)+\langle\nabla V(t x), t\,x\rangle}{N}|u|^{2}\Big)dx\\
&\leq \int\limits_{\mathbb{R}^{N}}\Big(2F(u)-V_{\infty}|u|^{2}\Big)dx.\\
\endaligned
\end{equation}
So, for each $u\in \mathcal{O}$, \eqref{eq3-12} implies that
there exist two constants $C_{1},\,C_2>0$, independent
of $t$, such that
\begin{equation}\label{eq3-15}
-C_{1}\leq \int\limits_{\mathbb{R}^{N}}\Big(2F(u)-
+\frac{NV(tx)+\langle\nabla V(t x), tx\rangle}{N}\,|u|^{2}
\Big)dx
\leq C_{2}.
\end{equation}
Thus,  together with \eqref{eq3-6} and \eqref{eq3-15},
we infer that there exists a $t>0$ sufficiently small such that $\theta'(t)>0$.
Therefore, by the  continuity of function $\theta'$, there exists at least one
$t_{u}>0$ such that $\theta'(t_{u})=0$, which means that $u(\frac{x}{t_{u}})\in \mathcal{P}$.

\vspace{3mm}

For the uniqueness of $t_{u}$, we define the function
$$\phi(t)=\int\limits_{\mathbb{R}^{N}}\Big(2F(u)
-\frac{NV(tx)+\langle\nabla V(t x), tx\rangle}{N}|u|^{2}\Big)dx,\quad\forall t>0.$$
Then the function $\phi$ is well-defined and belong to $C^{1}(\mathbb{R}^{+},\mathbb{R})$ under our assumptions.

Calculating the derivative of $\phi$ and applying the conditions $(V_2)$ and $(V_4)$, we obtain
\begin{equation*}\label{eq3-17}
\aligned
\phi'(t)
&=-\int\limits_{\mathbb{R}^{N}}\Big(\langle\nabla V(t x), x\rangle|u|^{2}
+\frac{x\cdot H(t x)\cdot (t x)}{N}|u|^{2}+\frac{\langle\nabla V(tx), x\rangle}{N}|u|^{2}\Big)dx\\
&=-\frac{1}{t}\int\limits_{\mathbb{R}^{N}}\Big(\langle\nabla V(tx),(tx)\rangle
+\frac{(tx)\cdot H(tx)\cdot (t x)+\langle\nabla V(tx),(tx)\rangle}{N}\Big)|u|^{2}dx\\
&=-\frac{1}{t}\int\limits_{\mathbb{R}^{N}}\Big(\langle\nabla V(tx),(tx)\rangle
+\frac{(tx)\cdot H(tx)\cdot (tx)}{N}\Big)|u|^{2}dx\\
&-\int\limits_{\mathbb{R}^{N}}\frac{1}{N}\langle\nabla V(tx),(tx)\rangle|u|^{2}dx\\
&>0,
\endaligned
\end{equation*}
which implies that $\phi$ is strictly increasing with respect to $t$.
Then, for every fixed $u\in \mathcal{O}$,  there exists a unique $t\in (0\,,\,+\infty)$
such that
$$\frac{N-2s}{2}\int\limits_{\mathbb{R}^{N}}|\xi|^{2s}|\hat{u}(\xi)|^{2}d\xi
=\frac{N}{2}t^{2s}\phi(t).$$
Then, by \eqref{eq3-6}, we now that there exists a unique $t\in (0\,,\,+\infty)$
such that $\theta'(t)=0$ for every $u\in \mathcal{O}$.
Hence, the uniqueness of $t_{u}$ is verified.

 \vspace{3mm}

Let us now define the operator $T:\,\mathcal{O}\mapsto \mathbb{R}^{+}$ by
 $$T[u]=t_{u}.$$
In order to prove the map $u\mapsto t_{u}$ is continuous,
It is sufficient to show the continuity of $T$.

Let $\{u_{n}\}_{n}$ be a sequence such that $u_{n}\in \mathcal{O}$
and $u_{n}\to u$ in $\mathcal{O}$ as $n\to\infty$.
We will show that
\begin{equation*}\label{eq3-18*}
T[u_{n}]\to T[u]\quad \text{ as} \,\,\,n\to \infty.
\end{equation*}

First we claim that $\{T[u_{n}]\}_{n}$ is bounded.
Indeed, the proof of the Proposition \ref{p3-1} and
$u_{n}(\frac{x}{T[u_{n}]})\in \mathcal{P}$ implies that
$\theta'(T[u_{n}])=0$, that is,
\begin{equation}\label{eq3-18}
\aligned
&(N-2s)\int\limits_{\mathbb{R}^{N}}|\xi|^{2s}|\hat{u}_{n}(\xi)|^{2}d\xi\\
&=NT[u_{n}]^{2s}\int\limits_{\mathbb{R}^{N}}\Big(2F(u_{n})
-\frac{NV(T[u_{n}] x)+\langle\nabla V(T[u_{n}] x), T[u_{n}] x\rangle}{N}\,|u_{n}|^{2}\Big)dx.\\
\endaligned
\end{equation}

Since $T[u_{n}]\in (0\,,\,+\infty)$ for all $n\in \mathbb{N}$, we can suppose by contradiction that
$T[u_{n}]\to +\infty$ as $n\to \infty$.
Taking the limit $n\to \infty$ in the right hand side of the equality
\eqref{eq3-18} we have
\begin{equation}\label{eq3-19}
\aligned
&\lim\limits_{n\to\infty}T[u_{n}]^{2s}\int\limits_{\mathbb{R}^{N}}
\Big(2F(u_{n})-\frac{NV(T[u_{n}] x)+\langle\nabla V(T[u_{n}] x), T[u_{n}] x\rangle}{N}\,|u_{n}|^{2}\Big)dx.\\
&=T[u_{n}]^{2s}\Big(\mathcal{F}(u_{n})+o_{n}(1)\Big)\\
&\to +\infty.\\
\endaligned
\end{equation}
On the other hand, for any $u_{n}\in \HR$,
we know that $$\int\limits_{\mathbb{R}^{N}}|\xi|^{2s}|\hat{u}_{n}(\xi)|^{2}d\xi<\infty,$$
which is contradict to \eqref{eq3-18} and \eqref{eq3-19}.
So $\{T[u_{n}]\}_{n}$ is bounded, and there exists a convergent subsequence of
$\{T[u_{n}]\}_{n}$ such that $T[u_{n}]\to T_{0}$ as $n\to \infty$.

By $T[u_{n}]\to T_{0}$, $u_{n}\to u$ as $n\to \infty$ and the continuity of $V$ and $F$, we get
\begin{equation}\label{eq3-20}
\int\limits_{\mathbb{R}^{N}}|\xi|^{2s}|\hat{u}_{n}(\xi)|^{2}d\xi
\to \int\limits_{\mathbb{R}^{N}}|\xi|^{2s}|\hat{u}(\xi)|^{2}d\xi;
\end{equation}
\begin{equation}\label{eq3-21}
\int\limits_{\mathbb{R}^{N}}V(T[u_{n}]x)\,|u_{n}|^{2}dx
\to \int\limits_{\mathbb{R}^{N}}V(T_{0}x)\,|u|^{2}dx;
\end{equation}
\begin{equation}\label{eq3-22}
\int\limits_{\mathbb{R}^{N}}\langle\nabla V(T[u_{n}]x), T[u_{n}]x\rangle\,|u_{n}|^{2}dx
\to
\int\limits_{\mathbb{R}^{N}}\langle\nabla V(T_{0}x), T_{0}x\rangle\,|u|^{2}dx,
\end{equation}
and
\begin{equation}\label{eq3-23}
\int\limits_{\mathbb{R}^{N}}F(u_{n})\,dx\to \int\limits_{\mathbb{R}^{N}}F(u)\,dx,
\end{equation}
as $n\to \infty$.
So, by \eqref{eq3-20}, \eqref{eq3-21}, \eqref{eq3-22}, \eqref{eq3-23} and \eqref{eq3-18}, we obtain
\begin{equation*}\label{eq3-24}
(N-2s)\int\limits_{\mathbb{R}^{N}}|\xi|^{2s}|\hat{u}(\xi)|^{2}d\xi
=NT_{0}^{2s}\int\limits_{\mathbb{R}^{N}}\Big(2F(u)
-\frac{NV(T_{0} x)+\langle\nabla V(T_{0} x),T_{0} x\rangle}{N}\,|u|^{2}\Big)dx.
\end{equation*}
Meanwhile, it follows from \eqref{eq3-18} that $T_{0}>0$ is such that $\theta'(T_{0})=0$, i.e., $u(\frac{x}{T_{0}})\in \mathcal{P}$.
The uniqueness of $t_{u}$ implies that $T[u]=T_{0}$.
Hence $T[u_{n}]\to T[u]$ as $n\to \infty$ in $\mathbb{R}^{+}$.
This completes the proof.
\end{proof}

For any $u\in \mathcal{P}_{\infty}$, one has that
$$\mathcal{F}(u)
=\int\limits_{\mathbb{R}^{N}}\Big(2 F(u)-V_{\infty} |u|^{2}\Big)dx
=\frac{N-2s}{N}\int\limits_{\mathbb{R}^{N}}|\xi|^{2s}|\hat{u}(\xi)|^{2}d\xi
>0.$$
Then we have the following result.
\begin{corollary}\label{co3-1}
For each $u\in \mathcal{O}$, there exists a unique $\widetilde{t}_{u}>0$ such that
$$u(\frac{x}{\widetilde{t}_{u}})\in \mathcal{P}_{\infty}.$$
Moreover, the function $u\mapsto \widetilde{t}_{u}$ such that $u(\frac{x}{\widetilde{t}_{u}})\in \mathcal{P}_{\infty}$ is continuous.
\end{corollary}

The proof of Corollary \ref{co3-1} is just similar to the proof of Proposition \ref{p3-1}.
An immediate consequence of Proposition \ref{p3-1} and Corollary \ref{co3-1}
is that for some $u\in \HR\setminus\{0\}$
can be projected on $\mathcal{P}$ or on $\mathcal{P}_{\infty}$
if and only if $\mathcal{F}(u)>0$.

 Now we have the following result.

\begin{proposition}\label{p3-2}
Let $u\in H^{s}(\mathbb{R}^{N})\setminus\{0\}$. Then
\begin{itemize}
\item[(i)]if $u\in \mathcal{P}$, there exists $\tilde{t}_{u}$ such that
$u(\frac{x}{\tilde{t}_{u}})\in \mathcal{P}_{\infty}$ and $ 0<\tilde{t}_{u}<1$;
\item[(ii)]if $u\in \mathcal{P}_{\infty}$,  there exists $t_{u}>0$ such that
$u(\frac{x}{t_{u}})\in \mathcal{P}$ and $t_{u}>1$.
\end{itemize}
\end{proposition}

\begin{proof}[\bf Proof.]
(i) For each $u\in \mathcal{P}$, by $(V_3)$, we have
\begin{equation}\label{eq3-28}
\aligned
\frac{N-2s}{N}\int\limits_{\mathbb{R}^{N}}|\xi|^{2s}|\hat{u}(\xi)|^{2}d\xi
&=\int\limits_{\mathbb{R}^{N}}\Big(2F(u)-\frac{NV(x)+\langle\nabla V(x), x\rangle}{N}|u|^{2}\Big)dx\\
&<\int\limits_{\mathbb{R}^{N}}\Big(2F(u)-V_{\infty}|u|^{2}\Big)dx=\mathcal{F}(u).\\
\endaligned
\end{equation}
Then $\mathcal{F}(u)>0$ for all $u\in \mathcal{P}$, that is, $u\in \mathcal{O}$.
Therefore, from Corollary \ref{co3-1},
there exists a unique $\tilde{t}_{u}>0$ such that $u(\frac{x}{\tilde{t}_{u}})\in \mathcal{P}_{\infty}$.

 Next we show that $\tilde{t}_{u}<1$.
 Note that $u(\frac{x}{\tilde{t}_{u}})\in \mathcal{P}_{\infty}$ yields
 \begin{equation}\label{eq3-29}
 \frac{N-2s}{2}\tilde{t}_{u}^{N-2s}\int\limits_{\mathbb{R}^{N}}|\xi|^{2s}|\hat{u}(\xi)|^{2}d\xi=
\frac{N}{2}\tilde{t}_{u}^{N}\int\limits_{\mathbb{R}^{N}}\Big(2F(u)-V_{\infty} |u|^{2}\Big)dx
=\frac{N}{2}\tilde{t}_{u}^{\,N}\mathcal{F}(u).
 \end{equation}
Taking account of  \eqref{eq3-28} and \eqref{eq3-29}, we infer $\tilde{t}_{u}$ satisfies
\begin{equation*}\label{eq3-30}
\tilde{t}_{u}^{\,\,2s}\mathcal{F}(u)=
\frac{N-2s}{N}\int\limits_{\mathbb{R}^{N}}|\xi|^{2s}|\hat{u}(\xi)|^{2}d\xi
<\mathcal{F}(u),
 \end{equation*}
which implies that $0<\tilde{t}_{u}<1$.

\vspace{5mm}

(ii) Let $u\in \mathcal{P}_{\infty}$, and then we have
\begin{equation}\label{eq3-25}
\aligned
\mathcal{F}(u)
=\int\limits_{\mathbb{R}^{N}}\Big(2 F(u)-V_{\infty} |u|^{2}\Big)dx
=\frac{N-2s}{N}\int\limits_{\mathbb{R}^{N}}|\xi|^{2s}|\hat{u}(\xi)|^{2}d\xi
>0,
\endaligned
\end{equation}
which implies that $u\in \mathcal{O}$ for all $u\in \mathcal{P}_{\infty}$. Thus,
from Proposition \ref{p3-1}, we have that there exists a $t_{u}\in (0\,,\,+\infty)$ such that
$u(\frac{x}{t_{u}})\in \mathcal{P}$.

 Next we will show that $t_{u}>1$.
 From $u(\frac{x}{t_{u}})\in \mathcal{P}$,
 one has that $\mathcal{J}(u(\frac{x}{t_{u}}))=0$. Therefore, we obtain
\begin{equation*}\label{eq3-26}
\aligned
&t_{u}^{N-2s}\Big(\frac{N-2s}{2}\int\limits_{\mathbb{R}^{N}}|\xi|^{2s}|\hat{u}(\xi)|^{2}d\xi+
\frac{t_{u}^{2s}}{2}\int\limits_{\mathbb{R}^{N}}\langle\nabla V(t_{u} x), t_{u} x\rangle|u|^{2}dx\\
&+\frac{N}{2}t_{u}^{2s}\int\limits_{\mathbb{R}^{N}}V(t_{u}x)|u|^{2}dx-
Nt_{u}^{2s}\int\limits_{\mathbb{R}^{N}}F(u)dx\Big)=0,\\
\endaligned\end{equation*}
that is,
\begin{equation}\label{eq3-27}
\frac{N-2s}{N}\int\limits_{\mathbb{R}^{N}}|\xi|^{2s}|\hat{u}(\xi)|^{2}d\xi
=t_{u}^{2s}\int\limits_{\mathbb{R}^{N}}\Big(2F(u)-\frac{\langle\nabla V(t_{u} x), t_{u} x\rangle+NV(t_{u}x)}{N}|u|^{2}\Big)dx.
\end{equation}
So, by condition $(V_3)$, \eqref{eq3-25} and \eqref{eq3-27}, we get
\begin{equation*}
\mathcal{F}(u)
<t_{u}^{2s}\int\limits_{\mathbb{R}^{N}}\Big(2F(u)-V_{\infty}|u|^{2}\Big)dx
=t_{u}^{2s}\,\mathcal{F}(u),
\end{equation*}
 which implies that $t_{u}>1$. This completes the proof.
\end{proof}

\section{A compactness result}
In this section we deal with the behavior of the (PS)-sequence of $\mathcal{I}$.
\begin{lemma}\label{l4-5}
Let $\{u_{n}\}\subset \HR$ be a \em{(PS)}-sequence of $\mathcal{I}$ constrained on $\mathcal{P}$,
 i.e., $u_{n}\in \mathcal{P}$ and
 \begin{equation}\label{eq4-1}
 \aligned
 &(a)\,\,\mathcal{I}(u_{n})\,\,\,\text{is bounded};\\
 &(b)\,\,\mathcal{I}'(u_{n})\to 0\,\,\,\text{as}\,\,\,n\to \infty.
 \endaligned
 \end{equation}
Then replacing $\{u_n\}_{n}$ by a subsequence, if necessary, there exists a solution $\bar{u}$ of Eq. {(FSE)},
a number $k\in \mathbb{N}\cup\{0\}$, $k$ functions $u^1,\, u^2,\,\cdot\cdot\cdot,\, u^k \in H^{s}(\mathbb{R}^{N})$
and $k$ sequences of points $\{y^{j}_{n}\}\subset \mathbb{R}^{N}$, $0\leq j\leq k$, such that
\begin{itemize}
\item[{(i)}] $|y_{n}^{j}|\to +\infty$, $|y_{n}^{j}-y_{n}^{i}|\to +\infty$, if $i\neq j$, $n\to \infty$;
\item[{(ii)}] $u_{n}-\sum\limits_{i=1}^{k}u^{i}(x-y_{n}^{i})\to \bar{u}$ in $E$;
\item[{(iii)}] $\mathcal{I}(u_n)\to \mathcal{I}(\bar{u})+\sum\limits_{i=1}^{k}\mathcal{I}_{\infty}(u^i)$.
\item[{(iv)}] $u^{j}$ are nontrivial weak solutions of Eq. $\text{\em(FSE)}_{\infty}$.
\end{itemize}
Moreover, we agree that in the case $k=0$ the above holds without $u^{j}$.
\end{lemma}

\begin{proof}[\bf Proof.]
We first observe that for any $u_{n}\in \mathcal{P}$,
\begin{equation*}
\aligned
\mathcal{I}(u_{n})
&=\frac{1}{2}\int\limits_{\mathbb{R}^{N}}|\xi|^{2s}|\hat{u}_{n}(\xi)|^{2}d\xi
+\frac{1}{2}\int\limits_{\mathbb{R}^{N}}V(x)|u_{n}|^{2}dx
-\int\limits_{\mathbb{R}^{N}}F(u_{n})dx\\
&=\frac{s}{N}\int\limits_{\mathbb{R}^{N}}|\xi|^{2s}|\hat{u}_{n}(\xi)|^{2}d\xi-
\frac{1}{2N}\int\limits_{\mathbb{R}^{N}}\langle \nabla V(x),x \rangle|u_{n}|^{2}dx\\
&\geq \frac{s}{N}\int\limits_{\mathbb{R}^{N}}|\xi|^{2s}|\hat{u}_{n}(\xi)|^{2}d\xi.
\endaligned
\end{equation*}
Hence, $\mathcal{I}(u_{n})$ bounded implies $\|u_{n}\|_{\dot{H}^{s}(\mathbb{R}^{N})}$ is bounded.
By Sobolev embedding theorem, it follow that $\|u_{n}\|_{L^{2^*_{s}}(\mathbb{R}^{N})}$ is also bounded.

Now by \eqref{eq1-2} and Lemma \ref{le2-2}, there exists $C_{1}>0$ such that
\begin{equation*}\label{eq4-2}
\aligned
\mathcal{I}(u_{n})
&\geq \frac{1}{2}\|u_{n}\|^{2}_{E}-\frac{\varepsilon}{2}\int\limits_{\mathbb{R}^{N}}|u_n|^{2}dx
-C(\varepsilon)\int\limits_{\mathbb{R}^{N}}|u_n|^{2^*_{s}}dx\\
&\geq \frac{1}{2}\|u_{n}\|^{2}_{E}-C_1\frac{\varepsilon}{2}\|u_n\|^{2}_{E}
-C(\varepsilon)\|u_n\|^{2^*_{s}}_{L^{2^*_{s}}(\mathbb{R}^{N})}\\
&\geq \frac{1}{2}(1-\varepsilon C_{1})\|u_{n}\|^{2}_{E}-C(\varepsilon)\|u_n\|^{2^*_{s}}_{L^{2^*_{s}}(\mathbb{R}^{N})}.\\
\endaligned\end{equation*}
Hence, taking $\varepsilon>0$ sufficiently small  such that $1-\varepsilon C_{1}>0$,
then it is easy to see that  $\{u_{n}\}_{n}$ is bounded in $E$.

We now claim that
\begin{equation}\label{eq4-3}
\mathcal{I}'(u_{n})\to 0\quad \text{as}\,\,n\to \infty.
\end{equation}
In fact, from \eqref{eq4-1} (b) we have
\begin{equation}\label{eq4-4}
0=\mathcal{I}'|_{\mathcal{P}}(u_{n})=\mathcal{I}'(u_{n})-\mu_{n}\mathcal{J}'(u_{n}),
\end{equation}
for some $\mu_{n}\in\mathbb{R}$. Similar as in the proof of Proposition \ref{p2-2}, we get that
$\mu_{n}\to 0$ as $n\to \infty$. Moreover, by the boundedness of $\{u_{n}\}_{n}$
and $\mathcal{J}'(u_{n})$ belongs to  $ C^{1}$, we know that $\mu_{n}\mathcal{J}'(u_{n})\to 0$
as $n\to \infty$. So \eqref{eq4-3} follows from \eqref{eq4-4}.

On the other hand, since $u_{n}$ is bounded in $E$, there exists $\bar{u}\in E$ such that,
up to a subsequence,
\begin{equation}\label{eq4-5}
\aligned
&u_{n}\rightharpoonup \bar{u}\,\,\, \text{in}\,\,\,E;\\
&u_{n}\to \bar{u}\,\,\,  \text{in}\,\,\,L^{p}_{loc}(\mathbb{R}^{N}),\,\,\forall p\in [2\,,\, 2^*_{s});\\
&u_{n}(x)\to \bar{u}(x)\,\,\,  \text{a.e. in}\,\,\,\mathbb{R}^{N}.\\
\endaligned
\end{equation}
Then we deduce that $\mathcal{I}'(\bar{u})=0$, that is, $\bar{u}$ is a weak
solution of Eq. (FSE).

If $u_{n}\to \bar{u}$ strongly in $E$. we are done. So we assume that $\{u_{n}\}_{n}$
dose not converge strongly to $\bar{u}$ in $E$.
Set $$z_{n}^{1}=u_{n}-\bar{u},$$
and therefore, $z_{n}^{1}\rightharpoonup 0$ weakly in $E$.
According to the Brezis-Lieb Lemma \cite{ae1983}, we deduce
\begin{equation}\label{eq4-6}
\int\limits_{\mathbb{R}^{2N}}\frac{|u_{n}(x)-u_{n}(y)|^{2}}{|x-y|^{N+2s}}dxdy=
\int\limits_{\mathbb{R}^{2N}}\frac{|z^{1}_{n}(x)-z^{1}_{n}(y)|^{2}}{|x-y|^{N+2s}}dxdy+
\int\limits_{\mathbb{R}^{2N}}\frac{|\bar{u}(x)-\bar{u}(y)|^{2}}{|x-y|^{N+2s}}dxdy+o(1);
\end{equation}
and
\begin{equation}\label{eq4-6*}
\|u_{n}\|^{p}_{L^{p}(\mathbb{R}^{N})}=\|z_{n}^{1}\|^{p}_{L^{p}(\mathbb{R}^{N})}+\|\bar{u}\|^{p}_{L^{p}(\mathbb{R}^{N})}+o(1).
\end{equation}

Let us show that
\begin{equation}\label{eq4-7}
\int\limits_{\mathbb{R}^{N}}F(u_{n})dx=\int\limits_{\mathbb{R}^{N}}F(\bar{u})dx+\int\limits_{\mathbb{R}^{N}}F(z_{n}^{1})dx+o(1).
\end{equation}
Observe that, in view of the mean value theorem and $(f_1)$ we have
\begin{equation}\label{eq4-7*}\aligned
&|\int\limits_{\mathbb{R}^{N}}F(u_{n})dx-\int\limits_{\mathbb{R}^{N}}F(\bar{u})dx-\int\limits_{\mathbb{R}^{N}}F(z_{n}^{1})dx|\\
&\leq \int\limits_{\mathbb{R}^{N}}|F(u_{n})-F(\bar{u})|dx+\int\limits_{\mathbb{R}^{N}}|F(z_{n}^{1})|dx\\
&\leq \int\limits_{\mathbb{R}^{N}}|f(\bar{u}+t_{1}z_{n}^{1})||u_{n}-\bar{u}|dx+
\int\limits_{\mathbb{R}^{N}}|f(t_{2}z_{n}^{1})||z_{n}^{1}|dx\\
&\leq \int\limits_{\mathbb{R}^{N}}C(\varepsilon)\Big(|\bar{u}|+t_{1}|z_{n}^{1}|\Big)|z_{n}^{1}|dx
+\int\limits_{\mathbb{R}^{N}}C(\varepsilon)t_2|z_{n}^{1}||z_{n}^{1}|dx\\
&\leq C\|z_{n}^{1}\|^{2}_{E}\to 0\quad (\,\,\text{as}\,\,n\to\infty\,\,),
\endaligned
\end{equation}
 where $t_{1}\,,\,t_{2}\in (0,1)$ and $C(\varepsilon)>0$.
 Hence \eqref{eq4-7} follows from \eqref{eq4-7*}.

Moreover, since $z_{n}^{1}\rightharpoonup 0$ as $n\to \infty$
weakly in $E$ and since $V(x)\to V_{\infty}$ as $|x|\to +\infty$,
the locally compact embedding $H^{s}(\mathbb{R}^{N})\hookrightarrow L^{p}_{loc}(\mathbb{R}^{N})$, $p\in [2\,,\,2^{*}_{s})$
gives
$$\int\limits_{\mathbb{R}^{N}}(V(x)-V_{\infty})|z_{n}^{1}|^{2}dx\to 0\quad \text{as}\,\,\,n\to\infty.$$
Therefore, together with \eqref{eq4-6}, \eqref{eq4-6*} and \eqref{eq4-7}, we obtain
\begin{equation}\label{eq4-8}\aligned
\mathcal{I}_{\infty}(z^{1}_{n})
&=\frac{1}{2}\int\limits_{\mathbb{R}^{N}}\frac{|z_{n}^{1}(x)-z_{n}^{1}(y)|^{2}}{|x-y|^{N+2s}}dxdy
+\frac{1}{2}\int\limits_{\mathbb{R}^{N}}V_{\infty}|z_{n}^{1}|^{2}dx-\int\limits_{\mathbb{R}^{N}}F(z_{n}^{1})dx\\
&=\frac{1}{2}\int\limits_{\mathbb{R}^{N}}\frac{|u_{n}(x)-u_{n}(y)|^{2}}{|x-y|^{N+2s}}dxdy
+\frac{1}{2}\int\limits_{\mathbb{R}^{N}}V(x)|u_{n}|^{2}dx-\int\limits_{\mathbb{R}^{N}}F(u_{n})dx\\
&-\Big(\frac{1}{2}\int\limits_{\mathbb{R}^{N}}\frac{|\bar{u}(x)-\bar{u}(y)|^{2}}{|x-y|^{N+2s}}dxdy
+\frac{1}{2}\int\limits_{\mathbb{R}^{N}}V(x)|\bar{u}|^{2}dx-\int\limits_{\mathbb{R}^{N}}F(\bar{u})dx\Big)+o(1)\\
&=\mathcal{I}(u_{n})-\mathcal{I}(\bar{u})+o(1),
\endaligned\end{equation}
and  for all $v\in H^{}(\mathbb{R}^{N})$,
\begin{equation}\label{eq4-9}\aligned
&0=\langle \mathcal{I}' (u_{n})\,,\,v\rangle\\
&=\int\limits_{\mathbb{R}^{2N}}\frac{((\bar{u}+z_{n}^{1})(x)-(\bar{u}+z_{n}^{1})(y))(v(x)-v(y))}{|x-y|^{N+2s}}dxdy
+\int\limits_{\mathbb{R}^{N}}V(x)u_{n}vdx-\int\limits_{\mathbb{R}^{N}}f(u_n)\,vdx\\
&=\int\limits_{\mathbb{R}^{2N}}\frac{\bar{u}(x)-\bar{u}(y))(v(x)-v(y))}{|x-y|^{N+2s}}dxdy+
\int\limits_{\mathbb{R}^{2N}}\frac{(z_{n}^{1}(x)-z_{n}^{1}(y))(v(x)-v(y))}{|x-y|^{N+2s}}dxdy\\
&+\int\limits_{\mathbb{R}^{N}}V(x)\bar{u}vdx+\int\limits_{\mathbb{R}^{N}}V_{\infty}z^{1}_{n}vdx-\int\limits_{\mathbb{R}^{N}}f(\bar{u})vdx
-\int\limits_{\mathbb{R}^{N}}f(z_{n}^{1})vdx+o(1)\\
&=\langle\mathcal{I}'(\bar{u})\,,\,v\rangle+\langle \mathcal{I}'_{\infty}(z_{n}^{1})\,,\,v\rangle +o(1),\\
\endaligned\end{equation}
so that $\mathcal{I}_{\infty}'(z_{n}^{1})=o(1)$ in $E'$. Furthermore,
$$\langle\mathcal{I}_{\infty}'(z_{n}^{1})\,,\,z_{n}^{1}\rangle=o(1).$$

Set
$$\delta:=\limsup\limits_{n\to +\infty}\Big(\sup\limits_{y\in \mathbb{R}^{N}}\int\limits_{B_{1}(y)}|z^{1}_{n}(x)|^{2}dx\Big).$$
If $\delta=0$, then Lemma \ref{le2-3} gives $\|z_{n}^{1}\|_{L^{p}(\mathbb{R}^{N})}\to 0$
as $n\to \infty$ for every $p\in (2\,,\,2^*_{s})$, and we infer
$$
\|u_{n}-\bar{u}\|^{2}_{E}=\|z_{n}^{1}\|^{2}\to 0\quad \text{as}\,\,\,n\to\infty.
$$
Hence $u_{n}\to \bar{u}$ as $n\to +\infty$, and we get the assertion.

In another case $\delta>0$, passing to a subsequence, we can find a sequence $\{y_{n}^{1}\}_{n}\subset \mathbb{R}^{N}$ such that
$$\int\limits_{B_{1}(y_{n}^{1})}|z_{n}^{1}|^{2}dx=
\int\limits_{B_{1}(0)}|z_{n}^{1}(x+y_{n}^{1})|^{2}dx>\frac{\delta}{2}.$$

Let us  consider $z_{n}^{1}(\cdot+y_{n}^{1})$. Since $\{z_{n}^{1}(\cdot+y_{n}^{1})\}_{n}$ is bounded, we may assume that $z_{n}^{1}(\cdot+y_{n}^{1})\rightharpoonup w_{1}\neq0$ in $E$
and $z_{n}^{1}(\cdot+y_{n}^{1})(x)\to w_{1}(x)$ a.e. in $\mathbb{R}^{N}$.
But since $z_{n}^{1}\rightharpoonup 0$, $\{y_{n}^{1}\}$ must be unbounded and, up to a subsequence, we can assume that
$|y_{n}^{1}|\to +\infty$ as $n\to \infty$.
Using the compactness of the embedding theorem and the continuity of $\mathcal{I}'_{\infty}$,
 we obtain
$$\langle\mathcal{I}'_{\infty}(w_{1})\,,\,\varphi\rangle=0\,\,\,\text{for every}\,\,\,\varphi\in E.$$

Finally, let us set
$z_{n}^{2}(x):=z_{n}^{1}(x)-w^{1}(x-y_{n}^{1})$. Then it follows that $z_{n}^{2}\rightharpoonup 0$ in $E$,
and by the same argument applied to $\mathcal{I}_{\infty}$, we obtain
$$
\mathcal{I}_{\infty}(z_{n}^{2})=\mathcal{I}(u_{n})-\mathcal{I}(\bar{u})-\mathcal{I}_{\infty}(w_{1})+o(1).
$$
and
$$
\langle \mathcal{I}'_{\infty}(z_{n}^{2}),\varphi\rangle=
\langle \mathcal{I}'_{\infty}(z_{n}^{1}),\varphi\rangle-
\langle \mathcal{I}'_{\infty}(w_{1}),\varphi(x-y_{n}^{1})\rangle=o(1),
$$
uniformly for $\|\varphi\|_{E}=1$ as $n\to +\infty$.

Now, if $z_{n}^{2}\to 0$ in $E$, we are done.
 Otherwise $z_{n}^{2}\rightharpoonup 0$
but not strongly and we repeat the argument as above.
Iterating this procedure we obtain a sequence of point $y_{n}^{j}\in \mathbb{R}^{N}$
such that $|y_{n}^{j}|\to +\infty$, $|y_{n}^{j}-y_{n}^{i}|\to +\infty$ if $i\neq j$ as $n\to +\infty$,
and a critical point $w_{i}$ of $\mathcal{I}_{\infty}$  with
$z_{n}^{i+1}(x)=z_{n}^{i}(x)-w_{i}(x-y_{n}^{i})$, $i\geq 1$ such that
$$
\mathcal{I}_{\infty}(z_{n}^{i+1})=\mathcal{I}(u_{n})-\mathcal{I}(\bar{u})-\sum\limits_{j=1}^{i}\mathcal{I}_{\infty}(w_{j})+o(1),
$$
where $\mathcal{I}'_{\infty}(w_{i})=o(1)$. Since $\mathcal{I}_{\infty}(w_{i})\geq m_{\infty}$ for all $i$
and $\mathcal{I}(u_{n})$ is bounded, the procedure has to stop after finite steps.
\end{proof}

\begin{corollary}\label{co4-1}
Let $\{u_{n}\}_{n}$ be a $(PS)_{c}$ sequence. Then $\{u_{n}\}_{n}$ is relatively compact for all
$c\in (0\,,\,m_{\infty})$.
Moreover, if $\mathcal{I}(u_{n})\to m_{\infty}$, then either $\{u_{n}\}$ is relatively compact
or the result of Lemma \ref{l4-5} holds with $k=1$, and $u^{1}=w$,
where $w$ is the ground state positive solution of Eq. $\text{\em(FSE)}_{\infty}$.
\end{corollary}
\begin{proof}[\bf Proof.]
Let $\{u_{n}\}_{n}$ be a $(PS)_{c}$ sequence of functional $\mathcal{I}$.
Applying Lemma \ref{l4-5} we have $\mathcal{I}_{\infty}(u^{j})\geq m_{\infty}$ for all $j\in \mathbb{N}$.

If $\mathcal{I}(u_{n})\to c<m_{\infty}$, Lemma \ref{l4-5} (iii) gives that $k=0$,
and then $u_{n}\to \bar{u}$ strongly in $E$.

If $\mathcal{I}(u_{n})\to c=m_{\infty}$ and $\{u_{n}\}_{n}$
is not compact in $E$, then Lemma \ref{l4-5} (iii) implies that $k=1$ and $u^{1}=w$.
\end{proof}

\section{Nonexistence result}
Now we are ready to prove Theorem \ref{th1-1}.
\begin{proposition}\label{p5-1}
The relation $m=m_{\infty}$ holds and $m$ is not attained.
\end{proposition}
\begin{proof}[\bf Proof.]
Let us show that $m\geq m_{\infty}$.

For all $u\in \mathcal{P}$, in view of Proposition \ref{p3-2} (i),
there is a $t_{u}\in (0\,,\,1)$ such that $u(\frac{x}{t_{u}})\in \mathcal{P}_{\infty}$,
so we get
$$
\frac{N-2s}{2}t_{u}^{N-2s}\int\limits_{\mathbb{R}^{N}}|\xi|^{2s}|\hat{u}(\xi)|^{2}d\xi
=\frac{N}{2}t_{u}^{N}\int\limits_{\mathbb{R}^{N}}\Big(2F(u)-V_{\infty}|u|^{2}\Big)dx.
$$
Then we have
\begin{equation}\label{eq3-37**}\aligned
\mathcal{I}_{\infty}\Big(u(\frac{x}{t_{u}})\Big)
&=\frac{t_{u}^{N-2s}}{2}\int\limits_{\mathbb{R}^{N}}|\xi|^{2s}|\hat{u}(\xi)|^{2}d\xi+
\frac{t_{u}^{N}}{2}\int\limits_{\mathbb{R}^{N}}V_{\infty}|u|^{2}dx-
t_{u}^{N}\int\limits_{\mathbb{R}^{N}}F(u)dx\\
&=\Big(\frac{1}{2}-\frac{N-2s}{2N}\Big)t_{u}^{N-2s}\int\limits_{\mathbb{R}^{N}}|\xi|^{2s}|\hat{u}(\xi)|^{2}d\xi\\
&=\frac{s}{N}t_{u}^{N-2s}\int\limits_{\mathbb{R}^{N}}|\xi|^{2s}|\hat{u}(\xi)|^{2}d\xi.\\
\endaligned\end{equation}

By \eqref{eq2-1} and $(V_2)$, we have
\begin{equation}\label{eq3-37}
\aligned
\mathcal{I}(u)
&=\frac{1}{2}\int\limits_{\mathbb{R}^{N}}|\xi|^{2s}|\hat{u}(\xi)|^{2}d\xi+
\frac{1}{2}\int\limits_{\mathbb{R}^{N}}V(x)|u|^{2}dx-
\frac{1}{2}\int\limits_{\mathbb{R}^{N}}F(u)dx\\
&=\frac{s}{N}\int\limits_{\mathbb{R}^{N}}|\xi|^{2s}|\hat{u}(\xi)|^{2}d\xi
-\frac{1}{2N}\int\limits_{\mathbb{R}^{N}}\langle\nabla V(x),x\rangle|u|^{2}dx\\
&\geq\frac{s}{N}\int\limits_{\mathbb{R}^{N}}|\xi|^{2s}|\hat{u}(\xi)|^{2}d\xi.\\
\endaligned
\end{equation}

Taking account of \eqref{eq3-37**}, \eqref{eq3-37} and  $t_{u}\in (0\,,\,1)$,
we infer
$$\mathcal{I}(u)\geq \frac{s}{N}\int\limits_{\mathbb{R}^{N}}|\xi|^{2s}|\hat{u}|^{2}d\xi
\geq \frac{s}{N}t_{u}^{N-2s}\int\limits_{\mathbb{R}^{N}}|\xi|^{2s}|\hat{u}|^{2}d\xi=\mathcal{I}_{\infty}\Big(u(\frac{x}{t_u})\Big)$$
for all $u\in \mathcal{P}$. Then we conclude that $m\geq m_{\infty}$.

Let us now prove the opposite inequality $m\leq m_{\infty}$.
To do this, we consider the sequence $u_{n}=w_{n}(\frac{\cdot}{t_{n}})$,
where $w_{n}(\cdot)=w(\cdot-z_{n})$, and $w$ is the positive ground solution centered at zero of Eq. $\text{(FSE)}_{\infty}$,
$\{z_{n}\}_{n}\subset \mathbb{R}^{N}$ with $|z_{n}|\to +\infty$ as $n\to +\infty$ and $t_{n}:=t_{w_{n}}$.
We want to show that
\begin{equation}\label{eq3-35*}
 \lim\limits_{n\to\infty}\mathcal{I}(u_{n})=m_{\infty}.
\end{equation}

From Proposition \ref{p2-3} we have $w\in \mathcal{P}_{\infty}$ and $\mathcal{I}_{\infty}(w)=m_{\infty}$.
 By the translation invariance of the integrals,
we get $w_{n}\in \mathcal{P}_{\infty}$ and $I_{\infty}(w_{n})=m_{\infty}$, too.
 By Proposition \ref{p3-2} (ii), we already know that $t_{n}=t_{w_{n}}>1$,
 and thus $u_{n}\in \mathcal{P}$.
 Therefore, we have
\begin{equation*}\label{eq3-35}
\aligned
&|\mathcal{I}({u}_{n})-m_{\infty}|\\
&=|\mathcal{I}({u}_{n})-\mathcal{I}_{\infty}(w_{n})|\\
&=\Big|\frac{1}{2}\int\limits_{\mathbb{R}^{N}}|\xi|^{2s}|\hat{u}_{n}(\xi)|^{2}d\xi
+\frac{1}{2}\int\limits_{\mathbb{R}^{N}}V(x)|{u}_{n}|^{2}dx
-\int\limits_{\mathbb{R}^{N}}F({u}_{n})dx\\
&-\frac{1}{2}\int\limits_{\mathbb{R}^{N}}|\xi|^{2s}|\hat{w}_{n}|^{2}d\xi
-\frac{1}{2}\int\limits_{\mathbb{R}^{N}}V_{\infty}|w_{n}|^{2}dx
+\int\limits_{\mathbb{R}^{N}}F(w_{n})dx\Big|\\
&=\Big|\frac{t_{n}^{N-2s}}{2}\int\limits_{\mathbb{R}^{N}}|\xi|^{2s}|\hat{w}_{n}(\xi)|^{2}d\xi
+\frac{t^{N}_{n}}{2}\int\limits_{\mathbb{R}^{N}}V(t_{n}x)|w_{n}|^{2}dx
-t_{n}^{N}\int\limits_{\mathbb{R}^{N}}F(w_{n})dx\\
&-\frac{1}{2}\int\limits_{\mathbb{R}^{N}}|\xi|^{2s}|\hat{w}_{n}(\xi)|^{2}d\xi
-\frac{1}{2}\int\limits_{\mathbb{R}^{N}}V_{\infty}|w_{n}|^{2}dx
+\int\limits_{\mathbb{R}^{N}}F(w_{n})dx\Big|\\
&=\Big|\frac{t_{n}^{N-2s}-1}{2}\int\limits_{\mathbb{R}^{N}}|\xi|^{2s}|\hat{w}_{n}(\xi)|^{2}d\xi
-(t_{n}^{N}-1)\int\limits_{\mathbb{R}^{N}}F(w_{n})dx\\
&+\frac{t^{N}_{n}}{2}\int\limits_{\mathbb{R}^{N}}V(t_{n}x)|w_{n}|^{2}dx
-\frac{1}{2}\int\limits_{\mathbb{R}^{N}}V_{\infty}|w_{n}|^{2}dx\Big|\\
&=\Big|\frac{t_{n}^{N-2s}-1}{2}\int\limits_{\mathbb{R}^{N}}|\xi|^{2s}|\hat{w}(\xi)|^{2}d\xi
-(t_{n}^{N}-1)\int\limits_{\mathbb{R}^{N}}F(w)dx\\
&+\frac{t^{N}_{n}}{2}\int\limits_{\mathbb{R}^{N}}V(t_{n}x+z_{n})|w|^{2}dx
-\frac{1}{2}\int\limits_{\mathbb{R}^{N}}V_{\infty}|w|^{2}dx\Big|\\
&\leq \frac{|t_{n}^{N-2s}-1|}{2}\int\limits_{\mathbb{R}^{N}}|\xi|^{2s}|\hat{w}(\xi)|^{2}d\xi
+|t_{n}^{N}-1|\int\limits_{\mathbb{R}^{N}}F(w)dx\\
&+\frac{1}{2}\int\limits_{\mathbb{R}^{N}}|t^{N}_{n}V(t_{n}x+z_{n})-V_{\infty}|\cdot|w|^{2}dx.\\
\endaligned
\end{equation*}
In order to obtain \eqref{eq3-35*}, we just need to show that $t_{n}\to 1$ as $n\to \infty$.

By Proposition \ref{p3-2} (ii), we know that $t_{n}>1$.
First, we claim $\lim_{n\to \infty}t_{n}<+\infty$.
If not, there exists a subsequence $\{t_{n}\}_{n}\subset \mathbb{R}^{N}$
such that $t_{n}\to +\infty$ as $n\to \infty$. Then
\begin{equation*}\label{eq3-31}
\aligned
&\mathcal{I}(u_{n})
=\mathcal{I}(w(\frac{x-z_{n}}{t_n}))\\
&=\frac{t_{n}^{N-2s}}{2}\int\limits_{\mathbb{R}^{N}}|\xi|^{2s}|\hat{w}(\xi)|^{2}d\xi
+\frac{1}{2}\int\limits_{\mathbb{R}^{N}}V(x)|w(\frac{x-z_{n}}{t_{n}})|^{2}dx
-\int\limits_{\mathbb{R}^{N}}F(w(\frac{x-z_{n}}{t_{n}}))dx\\
&=\frac{t_{n}^{N-2s}}{2}\int\limits_{\mathbb{R}^{N}}|\xi|^{2s}|\hat{w}(\xi)|^{2}d\xi
+\frac{t_{n}^{N}}{2}\int\limits_{\mathbb{R}^{N}}V(t_{n}x+z_{n})|w|^{2}dx
-t_{n}^{N}\int\limits_{\mathbb{R}^{N}}F(w)dx\\
&=\frac{t_{n}^{N-2s}}{2}\int\limits_{\mathbb{R}^{N}}|\xi|^{2s}|\hat{w}(\xi)|^{2}d\xi
-\frac{t_{n}^{N}}{2}\int\limits_{\mathbb{R}^{N}}\Big(2F(w)-V(t_{n}x+z_{n})|w|^{2}\Big)dx\\
&=\frac{t_{n}^{N-2s}}{2}\int\limits_{\mathbb{R}^{N}}|\xi|^{2s}|\hat{w}(\xi)|^{2}d\xi
-\frac{t_{n}^{N}}{2}\int\limits_{\mathbb{R}^{N}}\Big(2F(w)-V_{\infty}|w|^{2}\Big)dx+o(1)\\
&=\frac{t_{n}^{N-2s}}{2}\int\limits_{\mathbb{R}^{N}}|\xi|^{2s}|\hat{w}(\xi)|^{2}d\xi
-\frac{t_{n}^{N}}{2}\mathcal{F}(w)+o(1).\\
\endaligned
\end{equation*}
So, we have from $\mathcal{F}(w)>0$ that $\lim_{n\to \infty}\mathcal{I}(u_{n})\to -\infty$,
which is contradict to $m\geq m_{\infty}$. The proof of the claim is completed.

Now let us  prove that $\lim_{n\to \infty}t_{n}=1$.
Indeed, since $w\in \mathcal{P}_{\infty}$ and $u_{n}\in \mathcal{P}$, we have
\begin{equation}\label{eq3-31*}
\aligned
(t_{n}^{2s}-1)\int\limits_{\mathbb{R}^{N}}F(w)dx
&=\frac{t_{n}^{2s}}{2}\int\limits_{\mathbb{R}^{N}}\frac{\langle \nabla V(t_{n}x+z_{n}),t_{n}x+z_{n}\rangle}{N}\,|w|^{2}dx\\
&+\frac{t_{n}^{2s}}{2}\int\limits_{\mathbb{R}^{N}}V(t_{n}x+z_{n})|w|^{2}dx
-\frac{1}{2}\int\limits_{\mathbb{R}^{N}}V_{\infty}|w|^{2}dx.\\
\endaligned
\end{equation}

Assume $t_{n}\to t_{0}$ as $n\to \infty$. By $(V_1)$, $(V_2)$, $(V_3)$ and
the Lebesgue Dominated Convergence Theorem, we have
\begin{equation}\label{eq3-31**}
(t_{0}^{2s}-1)\int\limits_{\mathbb{R}^{N}}F(w)dx= \frac{1}{2}\int\limits_{\mathbb{R}^{N}}(t_{0}^{2s}-1)V_{\infty}|w|^{2}dx+o(1).
\end{equation}

If $t_{0}\neq 1$, it follows  from \eqref{eq3-31**} and $w\in \mathcal{P}_{\infty}$ that
$$
\frac{N-2s}{2N}\int\limits_{\mathbb{R}^{N}}|\xi|^{2s}|\hat{u}(\xi)|^{2}d\xi=\int\limits_{\mathbb{R}^{N}}F(w)dx
-\frac{1}{2}\int\limits_{\mathbb{R}^{N}} V_{\infty}|w|^{2}dx=0,
$$
which  is contradict to $u\neq 0$. Now we have $t_{n}\to 1$ as $n\to \infty$,
that is, we get that $m=m_{\infty}$.

\vspace{3mm}
Finally, assume, by contradiction, that there exists $u_{0}\in \mathcal{P}$ such that $\mathcal{I}(u_{0})=m=m_{\infty}$.
By Proposition \ref{p3-2}, there is a $t_{u_0}\in (0\,,\,1)$ such that $u_{0}(\frac{x}{t_{u_0}})\in \mathcal{P}_{\infty}$.
Then, by  $(V_2)$ and \eqref{eq3-37}, for  $u_{0}(\frac{x}{t_{u_0}})\in \mathcal{P}_{\infty}$, we have
\begin{equation*}
\aligned
m_{\infty}
&\leq \mathcal{I}_{\infty}(u_{0}(\frac{x}{t_{u_0}}))=\frac{s}{N}t_{u_0}^{N-2s}\int\limits_{\mathbb{R}^{N}}|\xi|^{2s}|\hat{u_0}(\xi)|^{2}d\xi\\
&<\frac{s}{N}\int\limits_{\mathbb{R}^{N}}|\xi|^{2s}|\hat{u_0}(\xi)|^{2}d\xi\\
&\leq \mathcal{I}(u_{0})=m=m_{\infty}.
\endaligned
\end{equation*}
Therefore $m$ is not attained. This completes the proof.
\end{proof}

\begin{proof}[\bf {Proof of Theorem \ref{th1-1}}]
Theorem \ref{th1-1} is obtained just by the result in Proposition \ref{p5-1}.
\end{proof}

\section{Existence of a positive solution}
By Proposition \ref{p5-1}, we can only hope to find critical points of $\mathcal{I}$ at levels higher than $m_{\infty}$.
Next lemma provides a range of values greater than $m_{\infty}$ such that the (PS) property holds.

\begin{lemma}\label{l4-7}
The functional $\mathcal{I}$ satisfies the $(\text{PS})_d$ condition for all  $d\in (m_{\infty}\,,\,2\,m_{\infty})$.
\end{lemma}

\begin{proof}[\bf Proof.]
Let us consider a $(PS)_{d}$ sequence $\{u_{n}\}_{n}\subset \HR$.
 By Lemma \ref{l4-5}, up to a subsequence, we have
\begin{equation}\label{eq4-22}
d=\lim\limits_{n\to +\infty}\mathcal{I}(u_{n})=\mathcal{I}(\bar{u})+\sum\limits_{i=1}^{k}\mathcal{I}_{\infty}(u^{i}),
\end{equation}
where $\bar{u}$ is the weak limit of $\{u_{n}\}_{n}$,
 and $u^{i}$ is a weak solution of Eq. $\text{(FSE)}_{\infty}$ and $\mathcal{I}_{\infty}(u^i)\geq m_{\infty}$.

Thus, being $m_{\infty}<d<2\,m_{\infty}$, \eqref{eq4-22} implies $k<2$.
If $k=1$, there are two possibilities:
\begin{itemize}
\item[(i)]$\bar{u}\neq0$, and then $\mathcal{I}(\bar{u})\geq m_{\infty}$ implies that
$$2m_{\infty}\leq\mathcal{I}(\bar{u})+\mathcal{I}_{\infty}(u^1)=\lim\limits_{n\to +\infty}\mathcal{I}(u_{n})=d<2m_{\infty};$$
\item[(ii)] $\bar{u}=0$, and then $\mathcal{I}(\bar{u})=0$ and
$$d=\lim\limits_{n\to +\infty}\mathcal{I}(u_{n})=\mathcal{I}_{\infty}(u^1)\in (m_{\infty}\,,\,2\,m_{\infty}).$$
This is impossible because either $\mathcal{I}_{\infty}(u^1)=m_{\infty}$, or $\mathcal{I}_{\infty}(u^1)\geq 2m_{\infty}$.
\end{itemize}
Since both cases brings to a contradiction, we conclude that $k=0$ and
 $u_{n}\to \bar{u}$ as $n\to\infty$, that is, the functional $\mathcal{I}$ satisfies the $(PS)_{d}$ condition.
\end{proof}

Now we need to build a suitable min-max scheme for problem (FSE).
To do this, we first remind the definition of the barycenter of a function $u\in H^{s}(\mathbb{R}^{N})$, $u\neq 0$
given in \cite{gd2003}.
Set
$$\mu(u)(x)=\frac{1}{|B_1(0)|}\int\limits_{B_{1}(x)}|u(y)|dy,$$
$$\tilde{u}(x)=\Big[\mu(u)(x)-\frac{1}{2}\max\mu(u)(x)\Big]^{+}.$$
It follows that $\mu$ is a continuous function and
$\mu(u)\in L^{\infty}(\mathbb{R}^{N})$ and $\tilde{u}\in C(\mathbb{R}^{N})$.

Now define the barycenter of $u$ by
$$\beta(u)=\frac{1}{\|\tilde{u}\|_{L^{1}(\mathbb{R}^{N})}}\int\limits_{\mathbb{R}^{N}}x\cdot\tilde{u}(x)\,dx.$$

Since $\tilde{u}$ has compact support, $\beta(u)$ is well-defined.
Moreover the  following properties hold:
\begin{itemize}
\item[(1)] $\beta$ is a continuous function in $H^{s}(\mathbb{R}^{N})\setminus\{0\}$;
\item[(2)] If $u$ is a radial function, then $\beta(u)=0$;
\item[(3)] For all $t\neq 0$ and $u\in H^{s}(\mathbb{R}^{N})\setminus\{0\}$,  $\beta(tu)=\beta(u)$;
\item[(4)] For given $z\in \mathbb{R}^{N}$, define $u_{z}(x):=u(x-z)$, and then $\beta(u_{z})=\beta(u)+z$.
\end{itemize}

Let us define
$$b:=\inf\{\mathcal{I}(u):\,u\in \mathcal{P},\,\beta(u)=0\}.
$$

It is clear that $b\geq m$. Moreover, the following result is true.
\begin{lemma}\label{l4-8*}
$b>m$.
\end{lemma}

\begin{proof}[\bf Proof.]
Clearly $b\geq m$. Suppose by contradiction that $b=m$.
Then there exists $\{u_{n}\}_{n}$ such that $u_{n}\in \mathcal{P}$, $\beta(u_{n})=0$ and
\begin{equation}\label{eq4-26*}
\mathcal{I}(u_{n})\to b=m=m_{\infty}\,\,\, \text{as}\,\,\,n\to\infty.
\end{equation}

By the Ekeland's variational principle, there exists a sequence
$\{w_{n}\}_{n}\subset \mathcal{P}$ such that
 $\mathcal{I}(w_{n})\to m_{\infty}$,
 $\mathcal{I}'|_{\mathcal{P}}(w_{n})\to 0$ and $\|w_{n}-u_{n}\|_{E}\to 0$.
Since $\mathcal{P}$ is a natural constraint of functional $\mathcal{I}$,
by an analogous argument in Lemma \ref{l4-5}
we deduce that $\mathcal{I}'(w_{n})\to 0$ and $\{w_{n}\}_{n}$ is bounded in $E$.
Moreover, $\{u_{n}\}$ is bounded. Since it is easily seen that
$\mathcal{I}''$ maps bounded set into bounded set,
the mean value theorem allows us to conclude that
\begin{equation}\label{eq4-26**}
\mathcal{I}'(u_{n})\to 0\quad \text{ as} \,\,\,n\to \infty.
\end{equation}

Taking account of \eqref{eq4-26*} and \eqref{eq4-26**},
Corollary \ref{co4-1} and Proposition \ref{p5-1} tell us that
$$u_{n}=w(x-z_{n})+o(1),$$
where $|z_{n}|\to +\infty$ as $n\to\infty$ for $z_{n}\in \mathbb{R}^{N}$,
and $w$ is the positive radially symmetric
solution of Eq. $\text{(FSE)}_{\infty}$.

By a translation, we obtain
\begin{equation*}\label{eq4-26}
u_{n}(x+z_{n})=w(x)+o(1).
\end{equation*}

Now computing the barycenter of both terms, we get
\begin{equation*}\label{eq4-27}
\beta(w(x)+o(1))=\beta(u_{n}(x+z_{n}))=\beta(u_{n})-z_{n}=-z_{n},
\end{equation*}
and
\begin{equation*}\label{eq4-28}
\beta(w(x)+o(1))\to \beta(w(x))=0\,\,\,(\text{since}\,\,\,w\,\,\,\text{is a radial function}).
\end{equation*}
That is a contradiction since $|z_{n}|\to +\infty$ as $n\to \infty$.
 Therefore, we must have $b>m$.
\end{proof}

Now let us  define the operator
$\Gamma:\,\mathbb{R}^{N}\to \mathcal{P}$ as
$$\Gamma[z](x)=w(\frac{x-z}{t_{z}}),\quad t_{z}>0,$$
where $w$ is the positive, radially symmetric and ground state solution  of Eq. $\text{(FSE)}_{\infty}$,
and $t_{z}$ is chosen such that $w(\cdot-z)$ projects  onto the Pohozaev manifold $\mathcal{P}$.
Moreover, we have
\begin{equation}\label{eq4-30*}
\aligned
\mu\Big(w(\frac{\cdot-z}{t_z})\Big)(x)
&=\frac{1}{|B_{1}(0)|}\int\limits_{B_{1}(x)}|w(\frac{y-z}{t_z})|dy\\
&=\frac{1}{|B_{1}(0)|}\int\limits_{B_{1}(x-z)}|w(\frac{y}{t_z})|dy
=\mu\Big(w(\frac{\cdot}{t_z})\Big)(x-z),\\
\endaligned\end{equation}
and
\begin{equation}\label{eq4-30**}
\aligned
\tilde{w}(\frac{\cdot-z}{t_{z}})(x)
&=\Big[\mu(w)\Big(w(\frac{\cdot-z}{t_z})\Big)(x)-\frac{1}{2}\max\mu\Big(w(\frac{\cdot-z}{t_z})\Big)(x)\Big]^{+}\\
&=\Big[\mu\Big(w(\frac{\cdot}{t_z})\Big)(x-z)-\frac{1}{2}\max\mu\Big(w(\frac{\cdot}{t_z})\Big)(x-z)\Big]^{+}
=\tilde{w}(\frac{\cdot}{t_{z}})(x-z).\\
\endaligned
\end{equation}

By using the properties of the barycenter and \eqref{eq4-30*}, \eqref{eq4-30**},
 we get
\begin{equation}\label{eq4-30}
\aligned
\beta(\Gamma[z])
&=\beta(w(\frac{x-z}{t_z}))
=\frac{1}{\|\tilde{w}(\frac{x-z}{t_z})\|_{L^1}}\int\limits_{\mathbb{R}^{N}}x\cdot \tilde{w}(\frac{x-z}{t_z})dx\\
&=\frac{1}{\|\tilde{w}(\frac{x-z}{t_z})\|_{L^1}}\int\limits_{\mathbb{R}^{N}}x\cdot \tilde{w}(\frac{\cdot}{t_{z}})(x-z)dx\\
&=\frac{1}{\|\tilde{w}(\frac{x}{t_z})\|_{L^1}}\int\limits_{\mathbb{R}^{N}}(x+z)\cdot \tilde{w}(\frac{\cdot}{t_{z}})(x)dx\\
&=\frac{1}{\|\tilde{w}(\frac{x}{t_z})\|_{L^1}}\int\limits_{\mathbb{R}^{N}}x\cdot \tilde{w}(\frac{\cdot}{t_{z}})(x)dx
+\frac{1}{\|\tilde{w}(\frac{x}{t_z})\|_{L^1}}\int\limits_{\mathbb{R}^{N}}z\cdot \tilde{w}(\frac{\cdot}{t_{z}})(x)dx\\
&=\beta(w(\frac{\cdot}{t_{z}}))+\frac{z}{|\tilde{w}|_{L^1}}\int\limits_{\mathbb{R}^{N}}\tilde{w}(\frac{x+z}{t_z})dx\\
&=0+z=z.\\
\endaligned
\end{equation}

Moreover, we obtain

\begin{lemma}\label{l4-10}
$\lim_{|z|\to +\infty}\mathcal{I}(\Gamma[z])\to m_{\infty}$.
\end{lemma}

\begin{proof}[\bf Proof.]
Indeed , since $\Gamma[z]\in \mathcal{P}$,  the functional $\mathcal{I}$
can be written as
\begin{equation*}\label{eq4-31}
\aligned
\mathcal{I}(\Gamma[z])
&=\frac{1}{2}\int\limits_{\mathbb{R}^{N}}|\xi|^{2s}|\widehat{\Gamma[z]}|^{2}d\xi
+\frac{1}{2}\int\limits_{\mathbb{R}^{N}}V(x)|\Gamma[z]|^{2}dx
-\int\limits_{\mathbb{R}^{N}}F(\Gamma[z])dx\\
&=\frac{s}{N}\int\limits_{\mathbb{R}^{N}}|\xi|^{2s}|\widehat{\Gamma[z]}|^{2}d\xi
-\frac{1}{2N}\int\limits_{\mathbb{R}^{N}}\langle \nabla V(x),x\rangle|\Gamma[z]|^{2}dx.
\endaligned
\end{equation*}

Moreover, from $w\in \mathcal{P}_{\infty}$ and
$$\mathcal{I}_{\infty}(w)=\frac{s}{N}\int\limits_{\mathbb{R}^{N}}|\xi|^{2s}|\hat{w}(\xi)|^{2}d\xi=m_{\infty},$$ we have that
\begin{equation*}\label{eq4-31}
\aligned
\mathcal{I}(\Gamma[z])
&=\frac{s}{N}\int\limits_{\mathbb{R}^{N}}|\xi|^{2s}|\widehat{\Gamma[z]}|^{2}d\xi
-\frac{1}{2N}\int\limits_{\mathbb{R}^{N}}\langle \nabla V(x),x\rangle|\Gamma[z]|^{2}dx\\
&=\frac{s}{N}t^{N-2s}_{z}\int\limits_{\mathbb{R}^{N}}|\xi|^{2s}|\hat{w}|^{2}d\xi
-\frac{t_{z}^{N}}{2N}\int\limits_{\mathbb{R}^{N}}\langle\nabla V(t_{z}x+z),(t_{z}x+z)\rangle|w|^{2}dx.
\endaligned
\end{equation*}

Using $t_{z}\to 1$ and $\langle \nabla V(t_{z}x+z),(t_{z}x+z)\rangle\to 0$ as $|z|\to +\infty$,
we get $\mathcal{I}(\Gamma[z])\to m_{\infty}$ if $|z|\to +\infty$. We get the assertion.
\end{proof}

\begin{lemma}\label{l4-9}
Assume that $(V_5)$ holds. Then
\begin{equation}\label{eq4-40}
\mathcal{I}(\Gamma[z])<2\,m_{\infty}.
\end{equation}
\end{lemma}

\begin{proof}[\bf Proof.]
First, we get
\begin{equation}\label{eq4-42}
\aligned
\mathcal{I}(\Gamma[z])
&=\frac{1}{2}\int\limits_{\mathbb{R}^{N}}|\xi|^{2s}|\widehat{\Gamma[z]}(\xi)|^{2}d\xi
+\frac{1}{2}\int\limits_{\mathbb{R}^{N}}V(x)|\Gamma[z]|^{2}dx-\int\limits_{\mathbb{R}^{N}}F(\Gamma[z])dx\\
&=\mathcal{I}_{\infty}(\Gamma[z])+\frac{1}{2}\int\limits_{\mathbb{R}^{N}}(V(x)-V_{\infty})|\Gamma[z]|^{2}dx\\
&=\frac{t_{z}^{N-2s}}{2}\int\limits_{\mathbb{R}^{N}}|\xi|^{2s}|\hat{w}(\xi)|^{2}d\xi
+\frac{t_{z}^{N}}{2}\int\limits_{\mathbb{R}^{N}}V_{\infty}|w|^{2}dx-t_{z}^{N}\int\limits_{\mathbb{R}^{N}}F(w)dx\\
&+\frac{1}{2}\int\limits_{\mathbb{R}^{N}}(V(x)-V_{\infty})|w(\frac{x-z}{t_{z}})|^{2}dx\\
&<t_{z}^{N}\mathcal{I}_{\infty}(w)+\frac{1}{2}\int\limits_{\mathbb{R}^{N}}(V(x)-V_{\infty})|w(\frac{x-z}{t_{z}})|^{2}dx\\
&=t_{z}^{N}\,m_{\infty}+\frac{1}{2}\int\limits_{\mathbb{R}^{N}}(V(x)-V_{\infty})|w(\frac{x-z}{t_{z}})|^{2}dx.\\
\endaligned
\end{equation}

By condition $(V_1)$, there is a $R>0$ such that $|V(x)-V_{\infty}|<\frac{\varepsilon}{2\|w\|^{2}_{L^{2}(\mathbb{R}^{N})}}$
for all $|x|>R$, then we have
\begin{equation*}\label{eq4-43}
\aligned
&\frac{1}{2}\int\limits_{\mathbb{R}^{N}}(V(x)-V_{\infty})|w(\frac{x-z}{t_{z}})|^{2}dx\\
&=\int\limits_{B_{R}(0)}(V(x)-V_{\infty})|w(\frac{x-z}{t_z})|^{2}dx
+\int\limits_{\mathbb{R}^{N}\setminus B_{R}(0)}(V(x)-V_{\infty})|w(\frac{x-z}{t_z})|^{2}dx\\
&\leq (|V_{\infty}|+\max\limits_{B_{R}(0)}|V(x)|)\int\limits_{B_{R}(0)}|w(\frac{x-z}{t_{z}})|^{2}dx
+ \frac{\varepsilon}{2\|w\|^{2}_{L^{2}(\mathbb{R}^{N})}} \int\limits_{\mathbb{R}^{N}}|w(\frac{x-z}{t_z})|^{2}dx\\
&= (|V_{\infty}|+\max\limits_{B_{R}(0)}|V(x)|)\int\limits_{B_{R}(-\frac{z}{t_z})}|w(x)|^{2}dx
+\frac{\varepsilon}{2}.\\
\endaligned
\end{equation*}
Pick $\varepsilon>0$ and choose $\bar{R}>0$ such that
$$\int\limits_{B_{R}(-\frac{z}{t_z})}|w(x)|^{2}dx\leq\frac{\varepsilon}{2(|V_{\infty}|+\max\limits_{B_{R}(0)}|V(x)|)}$$
for any $z\in \mathbb{R}^{N}$ and any $R<\bar{R}$. Hence
\begin{equation}\label{eq4-44}
\frac{1}{2}\int\limits_{\mathbb{R}^{N}}(V(x)-V_{\infty})|w(\frac{x-z}{t_{z}})|^{2}dx=o(1).
\end{equation}
Therefore, taking account of \eqref{eq4-42} and \eqref{eq4-44},
in order to prove \eqref{eq4-40} it is enough to show $t_{z}^{N}<2$
for all $z\in \mathbb{R}^{N}$. Let us observe that $t_{z}$ is chosen such that
\begin{equation*}
\aligned
\frac{N-2s}{2}t_{z}^{N-2s}\int\limits_{\mathbb{R}^{N}}|\xi|^{2s}|\hat{w}(\xi)|^{2}d\xi
&+\frac{t_{z}^{N}}{2}\int\limits_{\mathbb{R}^{N}}\langle   \nabla V(t_{z}x+z),(t_{z}x+z)\rangle|w|^{2}dx\\
&+\frac{Nt_{z}^{N}}{2}\int\limits_{\mathbb{R}^{N}}V(t_{z}x+z)|w|^{2}dx
=Nt_{z}^{N}\int\limits_{\mathbb{R}^{N}}F(w)dx\\
\endaligned
\end{equation*}
holds. By $t_{z}> 1$, $w\in \mathcal{P}_{\infty}$ and \eqref{eq4-44}, we deduce

\begin{equation*}
\aligned
\frac{N-2s}{2t^{2s}_{z}}\int\limits_{\mathbb{R}^{N}}|\xi|^{2s}|\hat{w}(\xi)|^{2}d\xi
&+\frac{1}{2}\int\limits_{\mathbb{R}^{N}}\langle   \nabla V(t_{z}x+z),(t_{z}x+z)\rangle|w|^{2}dx\\
&=\frac{N-2s}{2}\int\limits_{\mathbb{R}^{N}}|\xi|^{2s}|\hat{w}(\xi)|^{2}d\xi+o(1)>0,\\
\endaligned
\end{equation*}
that is,
$$
t_{z}^{2s}<\frac{(N-2s)\int_{\mathbb{R}^{N}}|\xi|^{2s}|\hat{w}(\xi)|^{2}d\xi}{-\int_{\mathbb{R}^{N}}\langle   \nabla V(t_{z}x+z),(t_{z}x+z)\rangle|w|^{2}dx}
<\frac{(N-2s)\|w\|^{2}_{H^{s}(\mathbb{R}^{N})}}{A\|w\|_{L^{2}(\mathbb{R}^{N})}^{2}}.
$$
Then by $(V_5)$, $t^{N}_{z}<2$ as desired.
\end{proof}

We are ready to prove Theorem \ref{th1-2}.

\begin{proof}[\bf {Proof of Theorem \ref{th1-2}}]
By Lemma \ref{l4-8*} and \ref{l4-10}, there exists $\bar{\rho}>0$ such that
for any $\rho\geq \bar{\rho}$,
\begin{equation}\label{eq4-32}
m_{\infty}<\max\limits_{|z|=\rho}\mathcal{I}(\Gamma[z])<b.
\end{equation}

In order to apply Linking Theorem \ref{th4-1}, we take
$$Q:=\Gamma(\overline{B}_{\bar{\rho}}(0))\,\,\,\text{and} \,\,\, S:=\{u\in \mathcal{P}:\,\beta(u)=0\}.$$

We claim that  $S$ and $\partial Q$ link, that is,
\begin{equation}\label{eq4-41}
\aligned
&(1)\,\, \partial Q\cap S=\emptyset;\\
&(2)\,\, h(Q)\cap S\neq \emptyset,\quad
\forall h\in \mathcal{H}:=\{h\in C(Q\,,\,\mathcal{P}):\,h|_{\partial Q}=\text{id}\}\\
\endaligned
\end{equation}
hold.

For \eqref{eq4-41} (1), observing that for any  $u\in \partial Q$ there exists $z$ such that $|z|=\bar{\rho}$
and $u=\Gamma[z]$, by \eqref{eq4-30}, we have $\beta(u)=\beta(\Gamma[z])=z\neq 0$.

For \eqref{eq4-41} (2), let us consider $h\in \mathcal{H}$ and define
$T:\,\overline{B}_{\bar{\rho}}(0)\to \mathbb{R}^{N}$ by
$$T(z)=\beta\circ h\circ \Gamma[z].$$
Then, $T$ is a continuous function. Moreover, for any $|z|=\bar{\rho}$,
we have $\Gamma[z]\in \partial Q$, and hence $h\circ \Gamma[z]=\Gamma[z]$
because $h|_{\partial Q}=\text{id}$. Thus,
$$T(z)=\beta\circ h \circ\Gamma [z]=\beta(\Gamma[z])=z.$$

By Brower fixed point theorem, we conclude that there exists $\bar{z}\in B_{\bar{\rho}}(0)$
such that $T(\bar{z})=\beta(h\circ \Gamma[\bar{z}])=0$,
 which means that $h\circ \Gamma[\bar{z}]\in S$.
Now we have $h(Q)\cap S\neq \emptyset$.

Furthermore, from the definition of $b$ and $Q$,  the inequalities \eqref{eq4-32} can be written as
$$b=\inf\limits_{S}\mathcal{I}>\max\limits_{\partial Q}\mathcal{I}.$$
Let us define $$d:=\inf\limits_{h\in \mathcal{H}}\max\limits_{u\in Q}\mathcal{I}(h(u)).$$
Then by \eqref{eq4-41} (2),  we get $d\geq b>m_{\infty}$.
Moreover, taking $h=\text{id}$ and by Lemma \ref{l4-9}, we obtain
$$d=\inf\limits_{h\in \mathcal{H}}\max\limits_{u\in Q}\mathcal{I}(h(u))\leq \max\limits_{u\in Q}\mathcal{I}(u)<2m_{\infty}.$$
This implies $m_{\infty}<d<2\,m_{\infty}$.
By Lemma \ref{l4-7}, the $(PS)_{d}$-condition holds in $(m_{\infty}\,,\,2\,m_{\infty})$,
and Linking Theorem tell us that $d$ is a critical value of the functional $\mathcal{I}$.
 This guarantees the existence of a nontrivial solution $u\in E$ of Eq. (FSE).
Because of hypotheses on $f$, and applying the maximum principle we may conclude as usual that $u$
is positive, which finishes the proof of Theorem \ref{th1-2}.
\end{proof}

\end{document}